\documentclass[final,leqno]{siamltex}

\usepackage{graphicx}
\usepackage{amssymb}
\usepackage{amsmath}
\usepackage{xspace}
\usepackage{epstopdf}
\usepackage{color}
\usepackage{dsfont}
\usepackage{bbm}
 \usepackage[colorlinks=true,linkcolor=blue]{hyperref}

 \usepackage[colorlinks=true,linkcolor=blue]{hyperref}
\hypersetup{
    bookmarks=true,         
    unicode=false,          
    pdftoolbar=true,        
    pdfmenubar=true,        
    pdffitwindow=false,     
    pdfstartview={FitH},    
    pdftitle={My title},    
    pdfauthor={Author},     
    pdfsubject={Subject},   
    pdfcreator={Creator},   
    pdfproducer={Producer}, 
    pdfkeywords={keyword1} {key2} {key3}, 
    pdfnewwindow=true,      
    colorlinks=true,       
    linkcolor=magenta,          
    citecolor=green,        
    filecolor=magenta,      
    urlcolor=cyan           
}
\usepackage{hyperref}

\usepackage{color,calc,xcolor}

\usepackage{ifthen}
\newcommand{\isdraft}{\boolean{true}} 
\renewcommand{\isdraft}{\boolean{false}} 
\ifthenelse{\isdraft}{
    \usepackage[color]{showkeys}
    \usepackage{xcolor}
}{}
\newcommand{\markupdraft}[2]{
    \ifthenelse{\equal{#1}{display}}{#2}{}
    \ifthenelse{\equal{#1}{color}}{\color{#2}}{}
}
\newcommand{\nnotecolored}[3][]{\markupdraft{display}{{\color{#2}\noindent[{\bf Note #1}: #3]}}}
\newcommand{\mathnote}[1]{\markupdraft{display}{{\color{brown}\noindent[{\bf Math note}: #1]}}}
\newcommand{\newcolored}[3][]{{\markupdraft{color}{#2}#3}
    \ifthenelse{\equal{#1}{}}{}{\markupdraft{display}{{\color{yellow!70!black}[#1]}}}} 

\providecommand{\del}[2][]{{\markupdraft{display}{{\color{red!20!yellow}[rmed: "#2"[#1]]}}}} 
\providecommand{\new}[2][]{\newcolored[#1]{blue}{#2}}
\providecommand{\rem}[2][]{\nnotecolored[#1]{green}{#2}} 
   
\providecommand{\nnote}[2][]{\nnotecolored[#1]{magenta}{#2}}   

\providecommand{\smalltodo}[2][]{\markupdraft{display}{{\color{cyan}~\noindent== small todo: #2 {\color{yellow}(#1)} ==}}}
\providecommand{\todo}[2][]{\markupdraft{display}{{\color{red}\noindent++TODO: #2 {\color{yellow}(#1)}++}}}

\ifthenelse{\isdraft}{}{\renewcommand{\markupdraft}[2]{}}

\newcommand{\anne}[1]{\rem[Anne]{\color{orange}#1}}

\newtheorem{assumption}{Assumption}{\bfseries}{\itshape}
{\bfseries}{\itshape}

\renewcommand{\t}{t}
\renewcommand{\k}{t}

\newcommand{\citecompanion}{\cite{methodology-paper}}

\newcommand{\UUU}{\ensuremath{\mathbb{U}}}
\newcommand{\Rplusstar}{\R^{+}_{>}}
\newcommand{\Rnstar}{\R^{n}_{\neq}}

\newcommand{\Ball}{B}
\newcommand{\Sphere}{\mathbb{S}}
\newcommand{\F}{\mathcal{F}}
\newcommand{\FF}{\mathcal{F}}
\newcommand{\GG}{\mathcal{G}}
\newcommand{\G}{G}

\newcommand{\Y}{\mathbf{Y}}

\newcommand{\ZZ}{\mathcal{Z}}

\newcommand{\LL}{\mathcal{L}}
\newcommand{\LLL}{\bar{\mathcal{L}}}
\newcommand{\Monotone}{\ensuremath{\mathcal{M}}}

\newcommand{\Nplus}{\NNN_{>}}

\newcommand{\BB}{\mathbf{B}}
\newcommand{\Surf}{\Sphere}
\newcommand{\aN}{\mathcal{N}}
\newcommand{\as}{\factonefifth}

\def\CR{{\rm CR}}
\def\PS{{\rm PS}}

\newcommand{\SSel}{\mathcal{O}rd}
\newcommand{\OOrd}{\mathcal{O}rd}

\newcommand{\xzero}{\mathbf{x}_{0}}

\newcommand{\plusES}{\ensuremath{(1+1)}\xspace}

\renewcommand{\dim}{n}
\newcommand{\Id}{I_n}

\newcommand{\B}{\mathcal{B}}
\newcommand{\R}{\mathbb{R}}
\newcommand{\N}{\ensuremath{\mathcal{N}}}

\newcommand{\NNN}{\ensuremath{{\mathbb{{N}}}}}

\newcommand{\vea}{\ensuremath{\mathbf{a}}}
\newcommand{\X}{\ensuremath{\mathbf{X}}}
\newcommand{\Z}{\ensuremath{\mathbf{Z}}}
\newcommand{\s}{\ensuremath{\mathbf{s}}}
\newcommand{\x}{\ensuremath{\mathbf{x}}}
\newcommand{\uu}{\ensuremath{\mathbf{u}}}

\newcommand{\y}{\ensuremath{\mathbf{y}}}
\newcommand{\yy}{\ensuremath{\mathbf{\tilde{y}}}}
\newcommand{\z}{\ensuremath{\mathbf{z}}}
\newcommand{\n}{\ensuremath{\mathbf{n}}}

\newcommand{\qq}{{q}}
\newcommand{\Uspace}{\ensuremath{\mathbb{U}^{p}}}
\newcommand{\U}{\ensuremath{\mathbf{U}}}
\newcommand{\Ut}{\ensuremath{\mathbf{U}_{\k}}}

\newcommand{\Zt}{\ensuremath{\mathbf{Z}_{\k}}}
\newcommand{\Ztt}{\ensuremath{\mathbf{Z}_{\k+1}}}
\newcommand{\Xt}{\ensuremath{\mathbf{X}_\k}}
\newcommand{\Xtt}{\ensuremath{\mathbf{X}_{\k+1}}}
\newcommand{\Yt}{\ensuremath{\mathbf{Y}_\k}}

\newcommand{\st}{\ensuremath{\sigma_\k}}
\newcommand{\stt}{\ensuremath{\sigma_{\k+1}}}

\newcommand{\etastar}{\ensuremath{\eta^{\star}}}

\newcommand{\factonefifth}{\gamma}

\newcommand{\xstar}{\ensuremath{\mathbf{x}^{\star}}}
\newcommand{\zero}{\ensuremath{\mathbf{0}}}

\newcommand{\muLeb}{\mu_{\rm Leb}}
\newcommand{\Rplus}{\R^{+}}

\newcommand{\Normal}{\mathcal{N}(0,I_{n})}

\newcommand{\KK}{K_{\ref{lem:integrability}}}

\newcommand{\acprs}{CB-SARS}
\newcommand{\asars}{SARS}

\newcommand{\scaleSI}{\rho}
\newcommand{\Sol}{\mathcal{S}ol}
\newcommand{\Perm}{\mathcal{S}}

\newcommand{\Isom}{\rm Homo}

\newcommand{\pp}{p_{{\mathcal{N}}}}

\title{Linear Convergence on Positively Homogeneous Functions of a Comparison Based Step-Size Adaptive Randomized Search: the (1+1) ES with Generalized One-fifth Success Rule 
        }

\author{Anne Auger\thanks{INRIA Saclay-Ile-de-France ({\tt anne.auger\_AT\_inria.fr}).}
        \and Nikolaus Hansen\thanks{INRIA Saclay-Ile-de-France ({\tt nikolaus.hansen\_AT\_inria.fr}) }
        }

\begin{document}

\maketitle

\begin{abstract}
In the context of unconstraint numerical optimization, this paper investigates the global linear convergence of a simple probabilistic derivative-free optimization algorithm (DFO). The algorithm samples a candidate solution from a standard multivariate normal distribution scaled by a step-size and centered in the current solution. This solution is accepted if it has a better objective function value than the current one. Crucial to the algorithm is the adaptation of the step-size that is done in order to maintain a certain probability of success. The algorithm, already proposed in the 60's, is a generalization of the well-known Rechenberg's $(1+1)$ Evolution Strategy (ES) with one-fifth success rule which was also proposed by Devroye under the name compound random search or by Schumer and Steiglitz under the name step-size adaptive random search. 

In addition to be derivative-free, the algorithm is function-value-free: it exploits the objective function only through comparisons. It belongs  to the class of comparison-based step-size adaptive randomized search (\acprs). For the convergence analysis, we follow the methodology developed in a companion paper for investigating linear convergence of \acprs: by exploiting invariance properties of the algorithm, we turn the study of global linear convergence on scaling-invariant functions into the study of the stability of an underlying normalized Markov chain (MC).

We hence prove global linear convergence by studying the stability (irreducibility, recurrence, positivity, geometric ergodicity) of the normalized MC associated to the $(1+1)$-ES. More precisely, we prove that starting from any initial solution and any step-size, linear convergence with probability one and in expectation occurs.
Our proof holds on unimodal functions that are the
composite of strictly increasing functions by positively homogeneous functions with degree $\alpha$ (assumed also to be continuously differentiable). This function class includes composite of norm functions but also non-quasi convex functions. Because of the composition by a strictly increasing function, it includes non continuous functions. We find that a sufficient condition for global linear convergence is the step-size increase on linear functions, a condition typically satisfied for standard parameter choices.

While introduced more than 40 years ago, we provide here the first proof of global linear convergence for the $(1+1)$-ES with generalized one-fifth success rule and the first proof of linear convergence for a \acprs\ on such a class of functions that includes non-quasi convex and non-continuous functions. Our proof also holds on functions where linear convergence of some \acprs\ was previously proven, namely convex-quadratic functions (including the well-know sphere function).

 \todo{in this paper plus comma paper: Link convergence on [give function properties] provided the step-size is increased on linear functions.}
\nnote{Concerning the question of the probability space - whether it makes sense to continue to use $\Pr$ while we define $P_{x}$ read Meyn Tweeedie page 56.}
\end{abstract}

\begin{keywords} 
linear convergence, derivative-free-optimization, comparison-based algorithm, function-value-free optimization , evolution strategies, step-size adaptive randomized search
\end{keywords}


\pagestyle{myheadings}
\thispagestyle{plain}
\markboth{A. AUGER AND N. HANSEN}{LINEAR CONVERGENCE OF THE (1+1)-ES WITH 1/5 SUCCESS RULE}


\section{Introduction}

\todo{In Stich et al. they cite the Giens paper for performance comparisons w.r.t. newuoa and bfgs - we might want to discuss indeed the local search performance - at least mention that the algorithms are seldomly presented as robust local search algorithms, however indeed good local cv property (linear cv towards local optima). Not true for many ``GA'' or PSO. In the mind of people, methods are good for global convergence ``only''. }

Derivative-free optimization (DFO) algorithms have the advantage to handle numerical optimization problems where the function $f: \R^{\dim} \mapsto \R$ to be WLG minimized can be seen as a black-box that is only able to return an objective function value $f(\x)$ for a given input vector $\x$. This context is particularly useful when dealing with many numerical optimization problems. Indeed, first, the function that needs to be optimized can result from a computer simulation where the source code might be too complex to exploit or might not be available to the person who has to do the optimization (this is typical in industry, where often only executables of the code are provided). Hence automatic differentiation to compute the gradient is not conceivable. Second, gradients can be non-exploitable because the function can be ``rugged'' that is noisy, very irregular, ...

Among DFO, we distinguish \emph{function-value-free} (FVF) algorithms that do not exploit the exact objective function value but only comparisons between candidate solutions. The Nelder-Mead simplex algorithm is one of the oldest deterministic FVF algorithm \cite{NelderMead:65}.
While the distinction between DFO and FVF algorithms is rarely made, it has some importance both in theory and practice because FVF algorithms are invariant to composing the objective function by a strictly increasing function and hence can be seen as more robust. 

We here focus on a particular class of probabilistic or randomized comparison-based (or FVF) algorithms  that adapt a mean vector (thought as favorite solution) and step-size.  A general framework for those methods has been formalized under the name \emph{comparison based step-size adaptive randomized search} (\acprs) \citecompanion. Those methods find their roots among the first papers published on randomized FVF algorithms in the 60's \cite{Matyas:1965,Schumer:68,Devroye:72,Rechenberg}. They were, later on, further developed in the Evolution Strategies (ES) community. The nowadays state-of-the-art Covariance Matrix Evolution Strategy (CMA-ES) where in addition to the step-size, a full covariance matrix is adapted (allowing to solve efficiently ill-conditioned problems) ensued from the developments on \acprs\ \cite{hansen2001}. Note that contrary to some common preconception, randomized FVF (in particular CMA-ES) are competitive also for ``local'' optimization and can show superior performance \new{compared to}\del{than} the standard BFGS or the NEWUOA \cite{newuoaReport:2007} algorithm on \emph{unimodal} functions provided they are significantly non-separable and non-convex \cite{auger:colloquegiens}.

We investigate the convergence of one of the earliest \acprs, introduced independently by Rechenberg under the name $(1+1)$-ES with one-fifth success rule \cite{Rechenberg}, by Devroye as the compound random search \cite{Devroye:72} and by Schumer and Steiglitz as step-size adaptive random search \cite{Schumer:68}. 
Formally, let $\Xt \in \R^{n}$ be the mean of a multivariate normal distribution  representing the favorite solution at the current iteration $t$. A new solution centered in $\Xt$ and following a multivariate normal distribution with standard deviation $\st$ (corresponding also to the step-size) is sampled:
\begin{equation}\label{eq:sampling}
\Xt^{1} = \Xt + \st \Ut^{1}
\end{equation}
where  $\Ut^{1}$ follows a standard multivariate normal distribution, i.e., $\Ut^{1} \sim \N(0, \Id)$. 
The new solution is evaluated on the objective function $f$ and compared to $\Xt$. If it is better than $\Xt$, in this case we talk about success, it becomes $\Xtt$, otherwise it is rejected:
 \begin{equation}\label{eq:update-mean}
\Xtt = \Xt + \st \Ut^{1} 1_{\{f(\Xt^{1}) \leq f(\Xt) \}} \enspace .
\end{equation}
As for the step-size, it is increased in case of success and decreased otherwise \cite{Schumer:68,Devroye:72,Rechenberg}. We denote $\factonefifth > 1$ the increasing factor and introduce a parameter $q \in \Rplusstar$ such that the factor for decrease equals $\factonefifth^{-1/q}$. Overall the step-size update reads
\begin{equation}\label{eq:update-ss}
\stt = \st \factonefifth 1_{\{f(\Xt^{1}) \leq f(\Xt) \}} + \st \factonefifth ^{-1/q} 1_{\{f(\Xt^{1})> f(\Xt) \}} \enspace .
\end{equation}
The idea to maintain a probability of success around $1/5$ was proposed in \cite{Schumer:68,Devroye:72,Rechenberg}. The constant $1/5$ is a trade-off between the asymptotic (in $n$) optimal success probability on the sphere function $f(\x) = \| \x \|^{2}$ where it is approximately  $0.27$ \cite{Schumer:68,Rechenberg} and the corridor function\footnote{The corridor function is defined as $f(\x) = \x_{1}$ for $- b < \x_{2} < b, \ldots -b  < \x_{n} < b$ (where $b >0$) otherwise $+ \infty$.} \cite{Rechenberg}. One implementation of the update of the step-size with target probability of success of $1/5$ is to set $-1/q = -1/4$\footnote{Assuming indeed a probability of success of $1/5$ and having set $q=4$ we find that $E[ \ln \stt | \st ] = \ln \st + \frac15 \ln \factonefifth + \frac45 \ln \factonefifth^{-1/4} = \ln \st $, i.e., the step-size is stationary.}. 
We call in the sequel the algorithm following equations \eqref{eq:sampling}, \eqref{eq:update-mean} and \eqref{eq:update-ss}, a $(1+1)$-ES with generalized one-fifth success rule and sometimes in short $(1+1)$-ES as there is no ambiguity for this paper that the step-size mechanism adopted is the generalized one-fifth success rule.

\acprs\ algorithms are observed to typically converge linearly towards local optima on a wide class of functions. Linear convergence of single runs is illustrated in Figure~\ref{fig:simul} for the $(1+1)$-ES on the simple sphere function $f(\x) = \| \x \|^{2}$. We observe that both the distance to the optimum $\| \Xt\|$ and the step-size $\st$ converge linearly at the same rate, in the sense that the logarithm of $\| \Xt \|$ (or $\st$) divided by $t$ converges to $- \CR$ (where $-\CR$ corresponds to the slope of the line observed in the second stage of the convergence).

Despite overwhelming empirical evidence of the linear convergence of \acprs\ and the fact that the methods are relatively old, few formal proofs of their linear convergence actually exist. A variant of the $(1+1)$-ES presented here was however studied by J\"agersk\"upper\footnote{In this variant, the step-size is kept constant for a period of several iterations before to increase or decrease depending on the observed probability of success during the period.} who proved on the sphere function and some convex-quadratic functions lower and upper bounds (on the time to reduce the error by a given fraction) that imply linear convergence \cite{Jens:2007,jens:gecco:2006,jens:2005,jens:tcs:2006}. The linear convergence of another \acprs\ using so-called self-adaptation as step-size adaptation mechanism was also proven on the sphere function \cite{TCSAnne04}.

We study in this paper the global linear convergence of the $(1+1)$-ES on a class of unimodal functions. More precisely convergence is investigated on functions $h$ that are the composition of a strictly increasing transformation $g$ by a positively homogeneous function with degree $\alpha$, $f$, i.e., satisfying $f(\rho(\x - \xstar)) = \rho^{\alpha} f(\x - \xstar) $ for any $\rho >0$, $\alpha >0$ and $\xstar$ that is the global optimum of the function (we assume that $f$ is strictly positive except in $\xstar$ where it can be zero). This class of function is a subset of scaling-invariant functions \citecompanion.

Under the assumptions that $f$ is continuously differentiable plus mild assumptions, we prove global linear convergence of the $(1+1)$-ES optimizing $h
$ provided $\gamma > 1$ and  the condition 
$
\frac12 \left( \frac{1}{\factonefifth^{\alpha}} + \factonefifth^{\alpha/q} \right) < 1
$
is satisfied.
(This latter condition translates that the step-size increases on a linear function in the sense that one over expected change to the $\alpha$ on a linear function is smaller $1$.)
More formally, under the conditions sketched above, assuming w.l.o.g. that $\xstar$ is zero, we prove the existence of $\CR > 0$ such that from any initial condition $(\X_{0}, \sigma_{0})$, almost surely
$$
\frac{1}{t} \ln \frac{\| \Xt \| }{ \| \X_{0} \|} \xrightarrow[t \to \infty]{} - \CR  \mbox{ and }\,\, \frac{1}{t} \ln \frac{\st}{\sigma_{0}} \xrightarrow[t \to \infty]{} - \CR \enspace
$$
hold. We provide a comprehensive expression for the convergence rate as
$$
\CR = - \ln \gamma   \left( \frac{q+1}{q} \PS - \frac1q	  \right) 
$$
where $\PS$ is the asymptotic probability of success.
We also prove that in expectation from any initial condition $(\X_{0},\sigma_{0}) = (\x,\sigma)$
$$
E_{\frac{\x}{\sigma}} \ln \frac{\| \Xtt \|}{\| \Xt \|} \xrightarrow[t \to \infty]{} - \CR \mbox{ and } \,\, E_{\frac{\x}{\sigma}} \ln \frac{\stt}{\st} \xrightarrow[t \to \infty]{} - \CR \enspace.
$$
We finally precise the speed of convergence for the step-size of the two previous equations. We prove a Central Limit Theorem associated to the first equation and then
prove that $ E_{\frac{\x}{\sigma}} \ln \frac{\stt}{\st}$ converges geometrically fast towards $-\CR$.

Our proof technique follows a methodology developed in \citecompanion\ exploiting the fact that the $(1+1)$-ES is a scale-invariant \acprs\ and that thus linear convergence on scaling-invariant functions can be turned into the stability study of the homogeneous Markov chain $\Zt=\Xt / \st$. 
More precisely we study the $\psi$-irreducibility, Harris-reccurence, positivity and geometric ergodicity of $(\Zt)_{t \in \NNN}$. We use for this, standard tools for the analysis of Monte Carlo Markov chains algorithms  and in particular Foster-Lyapunov drift conditions \cite{Tweedie:book1993}.

This paper is organized as follows. In Section~\ref{sec:NMC} we summarize the results from the companion paper \citecompanion\ setting the framework for the theoretical analysis, i.e., allowing us to define the normalized Markov chain that needs to be studied for proving the convergence. In addition, we define the objective functions under study  and set some first assumptions. In Section~\ref{sec:stability} we study the normalized chain $(\Zt)_{t \in \NNN}$ namely its $\varphi$-irreducibility, aperiodicity, investigate small sets and prove its geometric ergodicity  that constitutes the core part of the study. Using those results, we finally prove in Section~\ref{sec:linear-convergence} the global linear convergence of the $(1+1)$-ES with generalized one-fifth success rule almost surely and in expectation. We provide a comprehensive expression for the convergence rate. Last we discuss our findings in Section~\ref{sec:discussion}.

\todo{
\begin{itemize}
\item talk in the intro or somewhere else about Motivation w.r.t line search (see below).
\end{itemize}
}

\nnote{Literature review - motivation: QueryComplexityDFO.pdf : analyze complexity of comparison based for stochastic optimization (i.e., noise) motivate with DE algorithm. Cite 3 paper for bounds when opt. problem not noisy - including nesterov below (propose algorithm like EGS)

NesterovRandomGradientFreeConvex.pdf : analyze an EGS like algorithm - check carefully what is done there. Seems to  be considered as the god as only paper cited.

http://dev.related-work.net/arxiv:0912.3995 - several references.

Motivation wrt line search:

$\star$ adaptation of step-size use information already within the population - no additional assumption.

$\star$ perfect line search gives lower bound of $1 - \frac{1}{n}$--``tight'' as expected distance decrease optimally bracketed between $1-\frac1n$ and $1-0.5 \frac1n$.

General motivation (see Jens Hit-and-Run intro which is quite good): said in Nesterov and in Nemirovsky and Yudin: difficult to study.

Note: generally puzzled by the bound for noisy optimization - they are probably not the lower bounds I believe as on multiplicative noise we get linear convergence.

Randomized DFO are arguably more robust to noise, outliers or multi-modality than deterministic DFO due to their stochastic natures. We however focus here on a ``local'' convergence property deriving a linear convergence rate.
}

\paragraph*{Notations} We denote $\Normal$ a standard multivariate normal distribution, i.e., with mean vector $0$ and covariance matrix identity. Its density is denoted $\pp$. Given a set $C$ we denote its complementary $C^{c}$. We denote $\Rnstar$ the set $\R^{n}$ minus the null vector and $\Rplusstar$ denotes the set of strictly positive real numbers. The set of strictly increasing functions $g$ from $\R$ to $\R$ or a subset of $\R$ to $\R$ is denoted $\Monotone$. Given an objective function $f$ we denote by $\LL_{c}$ the level set $\{ \x, f(\x) = c\}$ and by $\LLL_{c}$ the corresponding sublevel set, i.e., $\LLL_{c}= \{ \x, f(\x) \leq c  \}$. We denote by $\NNN$ the set of natural numbers including $0$, i.e., $\NNN=\{0,1,\ldots \}$ and $\mathbb{N}_{>}$ the set $\{1, 2, \ldots \}$. The euclidian norm of a vector $\x$ is denoted $\| \x \|$. A ball of center $\x$ and radius $r$ is denoted $\BB(\x,r)$.

\section{Normalized Markov Chain and Objective Function Assumptions}\label{sec:NMC}

In this section we summarize the main results from \citecompanion\ allowing to define on the class of scaling-invariant functions the normalized Markov chain $\Xt/\st$. The study of the stability of this latter chain will imply the global linear convergence of the $(1+1)$-ES with generalized one-fifth success rule.

\subsection{The (1+1)-ES as a Comparison-Based Step-Size Adaptive Randomized Search}

We remind in this section that the $(1+1)$-ES is a \acprs\ after recalling the  general definition of a step-size adaptive randomized search (\asars) and a \acprs.

A \asars\ algorithm is identified to a sequence of random vectors $(\Xt,\st)_{t \in \NNN}$ where $\Xt \in \R^{n}$ and $\st \in \Rplusstar$. The vector $(\Xt,\st)$ is the state of the algorithm at iteration $t$ and $\Omega = \R^{n} \times \Rplusstar $ is its state space. Let $\UUU$ be a subset of $\R^{m}$  that is called sampling space and $\Uspace=\UUU \times \ldots \times \UUU$ for $p \in \Nplus$. Given $(\X_{0},\sigma_{0}) \in \Omega$, the sequence $(\Xt,\st)$ is inductively defined via
$$
(\Xtt,\stt) = \F^{f}((\Xt,\st),\Ut)
$$
where $\F$ is a measurable function and $(\Ut)_{t \in \NNN}$ is an independent and identically distributed (i.i.d.) sequence of random vectors of $\Uspace$. The objective function $f$ is also an input argument to the update function $\F$, however fixed over time, hence it is denoted as upper-script of $\F$. A comparison-based \asars\ is a particular case of \asars\ where candidate solutions are (i) sampled from $(\Xt,\st)$ using a solution function $\Sol$, (ii) evaluated on the objective function and ordered. The order of the candidate solutions is then \emph{solely} used for updating the state $(\Xt,\st)$ of the algorithm. Formally, let us define first the solution function and ordering function.
\begin{definition}[$\Sol$ function]\label{def:Sol} 
A $\Sol$ function used to create candidate solutions is a measurable function mapping $\Omega \times \UUU$ into $\R^{n}$, i.e.,
$$
\Sol: \Omega \times \UUU \mapsto \R^{n} \enspace.
$$
\end{definition}
\begin{definition}[$\OOrd$ function]\label{def:Ord} The ordering function $\OOrd$ maps $\R^{p}$ to $\mathfrak{S}(p)$, the set of permutations with $p$ elements and returns for any set of real values $(f_{1},\ldots,f_{p})$ the permutation of ordered indexes. That is $\Perm = \OOrd(f_{1},\ldots,f_{p}) \in \mathfrak{S}(p)$ where
$$
f_{\Perm(1)} \leq \ldots \leq f_{\Perm(p)} \enspace.
$$
When more convenient we might denote $\OOrd((f_{i})_{i=1,\ldots,p})$ instead of $\OOrd(f_{1},\ldots,f_{p}) $. \new{When needed for the sake of clarity, we might use the notations $\OOrd^{f}$ or $\Perm^{f}$ to emphasize the dependency in $f$.} \smalltodo{I might actually not need that}
\end{definition}

Given a permutation $\Perm \in \mathfrak{S}(p)$, the star operator $*$  defines the action of $\Perm$ on the coordinates of a vector $\U=(\U^{1},\ldots,\U^{p})$ belonging to $\Uspace$ as
\begin{equation}
\begin{aligned}
\mathfrak{S}(p) \times \Uspace  \to & \Uspace \\
(\Perm, \U) \mapsto & \Perm*\U = \left(\U^{\Perm(1)},\ldots,\U^{\Perm(p)} \right)  \enspace.
\end{aligned}
\end{equation}
A \acprs\ can now be defined using a solution function $\Sol$, the ordering function and the star operator.
\begin{definition}[\acprs\ minimizing $f:\R^{\dim} \to \R$]\label{def:SSAES} Let $p \in \Nplus$ and $\Uspace=\UUU \times \ldots \times \UUU$ where $\UUU$ is a subset of $\R^{m}$. Let $p_{\U}$ be a probability distribution defined on $\Uspace$ where each $\U$ distributed according to $p_{\U}$ has a representation $(\U^{1},\ldots,\U^{p})$ (each $\U^{i} \in \UUU$). Let $\Sol$ be a solution function as in Definition~\ref{def:Sol}. Let $\GG_{1}: \Omega \times \Uspace \mapsto  \R^{n} $ and $\GG_{2}:  \Rplusstar \times \Uspace \mapsto  \Rplus$ be two mesurable mappings and denote $\GG=(\GG_{1},\GG_{2})$.

A \acprs\ is determined by  the quadruplet $(\Sol,\GG,\Uspace,p_{\U})$ from which the recursive sequence $(\Xt,\st) \in \R^{n} \times \Rplusstar$ is defined via $(\X_{0},\sigma_{0}) \in \R^{\dim} \times \Rplusstar$ and for all $t$:
\begin{align}\label{eq:sol}
\Xt^{i} & = \Sol((\Xt,\st),\Ut^{i}) \,, i=1,\ldots,p \\\label{eq:perm}
\Perm & = \SSel(f(\Xt^{1}),\ldots,f(\Xt^{p})) \in \mathfrak{S}(p)\\\label{eq:G1}
\Xtt & = \GG_{1}\left( (\Xt,\st), \Perm*\Ut \right)\\\label{eq:G2}
\stt & = \GG_{2}\left(\st,  \Perm*\Ut \right)  
\end{align}
where $(\Ut)_{t \in \NNN}$ is an i.i.d.\ sequence of random vectors on \Uspace\ distributed according to $p_{\U}$, $\OOrd$ is the ordering function as in Definition~\ref{def:Ord}.
\end{definition}

In the next lemma we state that the $(1+1)$-ES with generalized one-fifth success rule is a \acprs\ and define its different components. The proof is immediate and hence omitted.
\begin{lemma}
The $(1+1)$-ES with generalized one-fifth success rule satisfies Definition~\ref{def:SSAES} with $p=2$, $\UUU=\R^{n}$, $\Uspace=\R^{n} \times \R^{n}$. Its solution function equals to
$$
\Sol: ((\x,\sigma),\uu) \in (\R^{n} \times \Rplusstar) \times \UUU) \mapsto \x + \sigma \uu \enspace.
$$
The sampling distribution is $p_{\U}(\uu_{1},\uu_{2})=\pp(\uu_{1}) \delta_{0}(\uu_{2})$ where $\pp$
is the density of a standard multivariate normal distribution and $\delta_{0}$ is the Dirac-delta function. The update function $\GG=(\GG_{1},\GG_{2})$ equals
\begin{equation}\label{eq:Goneplusone}
\GG((\x,\sigma),\y) = 
\left( \begin{smallmatrix}   \GG_{1}((\x,\sigma),\y) \\ 
\GG_{2}(\sigma,\y)
\end{smallmatrix} \right)
=  \left( \begin{smallmatrix}   \x +  \sigma  \y^{1} \\ 
\sigma  
\left( (\factonefifth - \factonefifth ^{-1/q}) 1_{\{ \y^{1} \neq 0 \}} + \factonefifth ^{-1/q}  \right)
\end{smallmatrix} \right)  \enspace.
\end{equation}
\end{lemma}
The solution and update functions associated to the $(1+1)$-ES have a specific structure that is useful for proving invariance properties of the algorithm. We state those properties in the following lemma and omit the proof which is also immediate.
\begin{lemma}\label{lem:PP}
Let $(\Sol,\GG,\Uspace,p_{\U})$ be the quadruplet associated to the \acprs\ $(1+1)$-ES with generalized one-fifth success rule. Then the following properties are satisfied:\\ For all $\x, \x_{0} \in \R^{n}$ for all $\sigma > 0$, for all $\uu \in \UUU$, $\y \in \Uspace$
\begin{align}\label{eq:ouap1}
\Sol((\x+\x_{0},\sigma),\uu) & = \Sol((\x,\sigma),\uu) + \x_{0} \\\label{eq:ouap2}
\GG_{1}((\x+\x_{0},\sigma),\y) & = \GG_{1}((\x,\sigma),\y) + \x_{0} \enspace.
\end{align}
For all $\alpha >0$, $(\x,\sigma) \in \Omega$, $\uu^{i} \in \UUU$, $\y \in \Uspace$
\begin{align}\label{eq:SIsol}
\Sol((\x,\sigma),\uu^{i}) & = \alpha \Sol \left( \left(\frac{\x}{\alpha}, \frac{\sigma}{\alpha}\right), \uu^{i}  \right) \\
\label{eq:SIG1}
\GG_{1}( (\x,\sigma),\y) & = \alpha \GG_{1}\left(\left(\frac{\x}{\alpha},\frac{\sigma}{\alpha}\right),\y\right)\\\label{eq:SIG2}
\GG_{2}(\sigma,\y) & = \alpha \GG_{2}\left(\frac{\sigma}{\alpha},\y\right)\enspace.
\end{align}
\end{lemma}
\subsection{Invariances}
As a direct consequence of the fact that the $(1+1)$-ES is comparison based, it is invariant to monotonically increasing transformations of the objective function. That is, for any $g \in \Monotone$, the sequence $(\Xt,\st)_{t \in \NNN}$ optimizing $g \circ f$ or optimizing $f$ are almost surely equal (see Proposition~2.4 in \citecompanion). In addition the $(1+1)$-ES is translation and scale-invariant as detailed below. 

Translation invariance implies identical behavior on a function $h(\x)$ or any of its translated version $\x \mapsto h(\x - \xzero)$.
It is formally defined for a \asars\  using a group homomorphism from the group $(\R^{n},+)$ to the group $(\mathcal{A}(\Omega),\circ)$, set of invertible mappings from the state space $\Omega$ to itself endowed with the function composition $\circ$. More precisely, a \asars\ is translation invariant if there exists a group homomorphism $\Phi \in \Isom( (\R^{n},+), (\mathcal{A}(\Omega), \circ) )$ such that for any objective function $f$, for any $\xzero \in \R^{n}$,  for any $(\x,\sigma) \in \Omega$ and \new{for any} $\uu \in \Uspace$
\begin{equation}\label{eq:TI}
\FF^{f(\x)}((\x,\sigma),\uu) = \underbrace{[\Phi(\xzero)]^{-1}}_{\Phi(-\xzero)} \left( \FF^{f\left( \x - \xzero \right)}(\Phi(\xzero)(\x,\sigma),\uu) \right)
\enspace.
\end{equation}
The $(1+1)$-ES is translation invariant and the group homomorphism associated equals $\Phi(\x_{0})(\x,\sigma) = (\x + \xzero,\sigma)$. This property is a consequence of \eqref{eq:ouap1} and \eqref{eq:ouap2} (see Proposition~2.7 in \citecompanion). Similarly, scale-invariance that translates that an algorithm has no intrinsic notion of scale is defined via homomorphisms from the group $(\Rplusstar,.)$ (where $.$ denotes the multiplication in $\R$) to the group $ (\mathcal{A}(\Omega), \circ) $. More precisely a \asars\ is scale-invariant  if there exists an homomorphism $\Phi \in \Isom((\Rplusstar,.),(\mathcal{A}(\Omega),\circ))$ such that for any $f$, for any $\alpha > 0$, for any $(\x,\sigma) \in \Omega$ and \new{for any} $\uu \in \Uspace$
\begin{equation}\label{eq:SI}
\FF^{f(\x)}((\x,\sigma),\uu) = \underbrace{[\Phi(\alpha)]^{-1}}_{\Phi(1/\alpha)} \left( \FF^{f\left(\alpha \x \right)}(\Phi(\alpha)(\x,\sigma),\uu) \right)
\enspace.
\end{equation}
The $(1+1)$-ES is scale-invariant and the group homomorphism associated equals $\Phi(\alpha) (\x,\sigma) = (\x/\alpha,  \sigma / \alpha) $. This property is a consequence of the properties \eqref{eq:SIG1} and \eqref{eq:SIG2} (see Proposition~2.9 in \citecompanion).

\subsection{Normalized Markov Chain on Scaling-Invariant Functions}\label{sec:defZ}

A class of functions that plays a specific role for \acprs\ are scaling invariant functions defined as: for all $\scaleSI > 0$, $\x,\y \in \R^{n}$
\begin{equation}\label{eq:scaling-invariant}
f(\scaleSI(\x - \xstar))  \leq f(\scaleSI(\y - \xstar)) \Leftrightarrow f(\x-\xstar) \leq f(\y - \xstar) \enspace,
\end{equation}
where $\xstar \in \R^{n}$. The latter function is said scaling-invariant w.r.t.\ $\xstar$. A linear function or any $g \circ f$ where $f$ is a norm and $g \in \Monotone$ are scaling-invariant. Also some non quasi-convex functions are scaling-invariant. Scaling-invariant functions are essentially unimodal, formally they do not admit any strict local extrema (see Proposition~3.2 in \citecompanion).

We assume given a scaling-invariant function w.r.t. $\xstar = 0$ (w.l.o.g.). Then, for a translation and scale-invariant \acprs\ defined by the quadruplet $(\Sol,\GG,\Uspace,p_{\U})$ where scale-invariance is a consequence of the properties \eqref{eq:SIsol}, \eqref{eq:SIG1} and \eqref{eq:SIG2}, the normalized sequence $(\Xt/\st)_{t \in \NNN}$ is an homogeneous Markov chain (Proposition~4.1 in \citecompanion). This Markov chain can be defined independently of $(\Xt,\st)$ provided $\Z_{0} = \frac{\X_{0}}{\sigma_{0}}$ via

\begin{align}
\Zt^{i} & = \Sol( (\Zt,1),\Ut^{i}), i=1,\ldots,p \\
\Perm & = \SSel(f(\Zt^{1}),\ldots,f(\Zt^{p}))\\
\Ztt & = \G(\Zt,\Perm*\Ut) 
\end{align}
where the transition function $\G$ equals for all $\z \in \R^{n}$ and $\y \in \Uspace$
\begin{equation}
G(\z,\y) = \frac{\GG_{1}((\z,1),\y)}{\GG_{2}(1,\y)}
 \enspace .
\end{equation}
According to the previous equation, the transition function $\G$ for the normalized chain $(\Zt=\frac{\Xt}{\st})_{t \in \NNN}$ associated to the $(1+1)$-ES on scaling-invariant functions is given by
\begin{equation*}
G(\z,\y)=\frac{\z +   \y^{1}}{ (\factonefifth - \factonefifth ^{-1/q}) 1_{\{ \y^{1} \neq 0 \}} + \factonefifth ^{-1/q} }
\end{equation*}
where the selected step $\y=(\y^{1},\y^{2})$ is according to the f-ranking of the solutions $\z+\uu^{1}$ and $\z+\uu^{2}$, 
i.e.,
$
f(\z + \y^{1}) \leq f(\z + \y^{2})
$.
However, since $\uu^{2} = 0$ (because $p_{\U}(\uu^{1},\uu^{2})=p_{\N}(\uu^{1})\delta_{0}(\uu^{2})$), $\y^{1} = \uu^{1} 1_{\{f(\z+\uu^{1}) \leq f(\z)\}} $. In addition since $\Ut^{1} \sim \N(0,\Id)$, the event ${\{ \Yt^{1} \neq 0 \}}$ is almost surely equal to the event $\{ \Yt^{1} = \Ut^{1} \}$ and hence almost surely equal to the event ${\{f(\Zt + \Ut^{1})\leq f(\Zt) \}}$.
Overall the Markov chain $(\Zt)_{t\in \NNN}$ satisfies $\Z_{0} = \frac{\X_{0}}{\sigma_{0}} $ and  given $(\U_{t}^{1})_{t \in \NNN}$ i.i.d with $\U_{t}^{1} \sim \N(0,\Id)$ 
\begin{equation}\label{eq:transitionZ}
\boxed{
\Ztt = \frac{\Zt +  \Ut^{1} 1_{\{f(\Zt + \Ut^{1})\leq f(\Zt) \}}   }{ (\factonefifth - \factonefifth ^{-1/q}) 1_{\{f(\Zt + \Ut^{1})\leq f(\Zt) \}}  + \factonefifth ^{-1/q}  } \enspace. }
\end{equation}
Following \citecompanion, we introduce the notation $\etastar$ for the step-size change, i.e.,
\begin{equation}\label{eq:sschangeETA}
\etastar = (\factonefifth - \factonefifth^{-1/q}) 1_{\{ f(\Xt^{1}) \leq f(\Xt) \}} + \factonefifth^{-1/q} 
\end{equation}
and remind that on scaling-invariant functions, the step-size change starting from $(\Xt,\st)$ is the same as the step-size change starting from $(\Zt,1)$ (see Eq.~(4.7) in \citecompanion) such that 
\begin{equation}\label{eq:ping}
\Ztt = \frac{\Zt +  \Ut^{1} 1_{\{f(\Zt + \Ut^{1})\leq f(\Zt) \}}}{\etastar} \enspace .
\end{equation}

%
\subsection{Objective Function Assumptions}\label{sec:scaling-inv}

We consider scaling-invariant functions formally defined by \eqref{eq:scaling-invariant} where in addition we assume that $\xstar=0$. This can be done w.l.o.g.\ because the $(1+1)$-ES is translation invariant. This assumption is sufficient to build the normalized Markov chain $(\Zt)_{t \in \NNN}$ (see Section~\ref{sec:defZ}). However for studying its stability, we will make further hypothesis on $f$. 
\del{We state and explain in this section those main hypothesis.}

We will consider a particular class of scaling-invariant functions, namely positively homogeneous functions. Formally a positively homogeneous function with degree $\alpha$ satisfies the following definition.
\begin{definition}\label{def:poshf}[Positively homogeneous functions]
A function $f:\R^\dim \mapsto \R$ is said positively  homogeneous with degree $\alpha$ if
for all $\rho >0$ and for all $\x \in \R^\dim$, $f(\rho \x) = \rho^{\alpha} f(\x)$. 
\end{definition}

Remark that positive homogeneity is not always preserved if $f$ is composed by a non-decreasing transformation. We will in addition make the following assumptions on the objective function:

\begin{assumption}\label{ass:f}
The function $f:\R^{\dim} \to \R$ is homogeneous with degree $\alpha$ and $f(\x) > 0$ for all $\x \neq 0$. 
\end{assumption}

This assumption implies that the function $f$ has a \emph{unique optimum} located w.l.o.g.\ in $0$ (if the optimum $\xstar$ is not in $0$, consider $\tilde f = f(\x - \xstar)$) as seen in the next lemma point (i). Remark that with this assumption, we exclude linear functions. 

In the next lemma, we state some properties of positive homogeneous functions satisfying Assumptions~\ref{ass:f}. We denote for $c \geq 0$, $\LLL_{c} = \{ \x, f(\x) \leq c \}$ the sublevel set of $f$ associated to $c$  and $\LL_{c}= \{ \x, f(\x) = c\}$ its level set. The hypersphere surface of radius $r$ is denoted $\Sphere_{r}$, that is $\Sphere_{r}=\{ \x , \| \x \| = r \}$.

\begin{lemma}\label{lem:propf}
Let $f$ be an homogeneous function with degree $\alpha >0$ and $f(\x) > 0$ for all $\x \neq 0$ and $f(\x)$ finite for every $\x \in \R^\dim$.  Then the following holds:
\begin{itemize}
\item[(i)] $\lim_{t \to 0} f(t \x)=0$ and assuming that $f(\zero)=0$, for all $\s \neq 0 $, the function $f_{\s}: t \in [0, + \infty[ \mapsto f(t \s)$ is continuous, strictly increasing and converges to $+ \infty$ when $t$ goes to $+ \infty$.
\item[(ii)] If $f$ is lower semi-continuous, then $\LLL_{c}$ is compact.
\end{itemize}
\end{lemma}

{\em Proof.}
(i) Since $f(t \x) = t^{\alpha} f(\x)$, fixing $\x$ and taking the limit for $t$ to zero we have that $\lim_{t \to 0} f(t \x)=0$. For any $\s$, the function $f_{\s}$ satisfies $f_{\s}(t) = t^{\alpha} f(\s)$. It is thus continuous on $[0,+\infty[$, strictly increasing and converges to infinity when $t$ goes to infinity. \\
(ii) Since $f$ is lower semi continuous, the inverse image of sets of the form $(-\infty,r]$ are closed sets. Hence $\LLL_{c} = f^{-1}((-\infty,c])$ is closed. Let us consider the surface $\Sphere_{r}$ for $r>0$. Since $f$ is lower semi-continuous, there exists $\x_{0} \in \Sphere_{r} $ such that $\inf_{\x \in \Sphere_{r}} f(\x) = f(\x_{0})$. Since $f(\x) > 0$ for $\x \neq 0$, $f(\x_{0}) : = m \neq 0$. Hence we have $\LLL_{m} \subset \Ball_{r}$. Because $f$ is homogeneous with degree $\alpha$, we have thus $\LLL_{m \sigma^{\alpha}} \subset \Ball_{r \sigma}$ for all $\sigma >0$. Hence, for any $c$ we can include $\LLL_{c}$ is a ball which proves that it is bounded and hence compact. \hfill \endproof

A positively homogeneous function satisfies for all $\x \neq 0$
\begin{equation}\label{eq:pos-hom-relation}
f(\x) = \| \x \|^{\alpha} f \left( \x / \| \x \| \right) \enspace.
\end{equation}
From this latter relation it follows that $f$ is continuous on $\Sphere_{1}=\{ \x \in \R^{\dim}, \| \x \|=1 \}$ if and only if $f$ is continuous on $\Rnstar$. 
\mathnote{Formally we use the continuity of $\z \to \| \z \|$, of $\z \to \frac{\z}{\| \z \|}$ and of the product.}
Assuming continuity on $\Rnstar$, we denote in the sequel $m$ the minimum of $f$ on $\Sphere_{1}$ and $M$ its maximum, that is
\begin{align}\label{eq:min-max}
m & =\min_{\z \in \Sphere_{1}}f(\z) \\\label{eq:min-max2}
M & = \max_{\z \in \Sphere_{1}} f(\z) \enspace.
\end{align}
The following lemma will be used several times when investigating the stability of the normalized Markov chain $\Z = (\Zt)_{t \in \NNN}$.
\begin{lemma}\label{lem:too}
Let $f$ satisfy Assumptions~\ref{ass:f} and $f$ be continuous on $\Sphere_{1}$. Then for all $\z \neq 0$
\begin{equation}\label{eq:LB-UB}
\| \z \| m^{1/\alpha} \leq f(\z)^{1/\alpha} \leq \| \z \| M^{1/\alpha}  \enspace,
\end{equation}
where $m$ and $M$ are defined in \eqref{eq:min-max} and \eqref{eq:min-max2}.
Hence, $f(\z) \to 0 $ when $\| \z \| \to 0$, $f(\z) \to \infty$ when $\| \z \|$ goes to $\infty$ and  $|\ln \| \z \|| f(\z)^{1/\alpha} \to 0$ when $\| \z \| \to 0$.
\end{lemma}
\begin{proof}
By homogeneity, for all $\z \neq 0$, we have $f(\z)=f\left(\| \z \| \frac{\z}{\|\z\|} \right) = \| \z \|^{\alpha}f\left(\frac{\z}{\|\z\|} \right)$. Since $f$ is continuous on the compact $\Sphere_{1}$, $m = \min_{\z \in \Surf_{1}}f(\z) > 0$ and $M =\max_{\z \in \Surf_{1}} f(\z) > 0$ and $M < \infty$.  We hence have
$$
\| \z \|^{\alpha} \underbrace{\min_{\z \in \Sphere_{1}}f(\z)}_{m}  \leq f(\z) \leq \| \z \|^{\alpha} \underbrace{\max_{\z \in \Sphere_{1}} f(\z)}_{M}
$$
and thus $\| \z \| m^{1/\alpha} \leq f(\z)^{1/\alpha} \leq \| \z \| M^{1/\alpha} $.
Since $\|\z\| |\ln \| \z \||  \to 0$ when $\|\z \| \to 0$, we hence obtain that $|\ln \| \z \|| f(\z)^{1/\alpha} \to 0$.
\end{proof}

The following lemma is a consequence of the previous one and will be useful in the sequel.
\begin{lemma}\label{lem:si}
Let $f$ satisfy Assumptions~\ref{ass:f} and $f$ be continuous on $\Surf_{1}$, for all $\rho > 0$, the ball centered in $0$ and of radius $\rho$ is included in the sublevel set of degree $\rho^{\alpha} M$, i.e.,
\begin{equation}\label{eq:bart1}
\BB(0,\rho) \subset \LLL_{\rho^{\alpha} M} \enspace.
\end{equation}
For all $K > 0$, the sublevel set of degree $K$ is included into the ball centered in $0$ and of radius $(K/m)^{\alpha}$, i.e.,
\begin{equation}\label{eq:bart2}
\LLL_{K} \subset \BB(0,{(K/m)^{\alpha}}) \enspace.
\end{equation}
\end{lemma}
\begin{proof}
From Lemma~\ref{lem:too} we have that for all $\z$, $m \| \z \|^{\alpha} \leq f(\z) \leq M \| \z \|^{\alpha}  $. Let $\z \in \BB(0,\rho)$, then $f(\z) \leq M \rho^{\alpha}$, i.e., $\z \in \LLL_{\rho^{\alpha} M}$. Let $\z \in \LLL_{K}$, then $f(\z) \leq K$ and hence $\| \z \| \leq (K/m)^{\alpha} $.
\end{proof}

Last, we remind the Euler's homogeneous function theorem.
\begin{theorem}[Euler's homogeneous function theorem]
Suppose that the function $f: \R^{n} \backslash \{0\} \mapsto \R$ is continuously differentiable. Then $f$ is positive homogeneous of degree $\alpha$ if and only if
$$
    \mathbf{x} \cdot \nabla f(\mathbf{x})= \alpha f(\mathbf{x}) \enspace. 
 $$
\end{theorem}
\noindent This theorem implies that if $f$ is positively homogeneous and continuously differentiable, if $f(\x) > 0$ for $\x \neq 0$ (i.e., Assumption~\ref{ass:f}), then 
\begin{equation}\label{prop-grad}
\nabla f(\x) \neq 0 \mbox{  for all  } \x \neq 0 \enspace.
\end{equation}

\smalltodo{check among the additional assumptions I use in the stability section, which one I want to include here.}

\subsection{State Space for the Normalized Markov Chain}

The state-space for the normalized Markov chain $\Z=(\Zt)_{t\in\NNN}$ is {\it a priori} $\R^{\dim}$. However if we start from $\Z_{0}=0$, we will stay in $0$ forever, i.e., $\Zt=0$ for all $t$. This is due to the fact that the $(1+1)$-ES cannot accept worse solutions and $0$ is the global optimum of $f$. This would then preclude the chain $\Z$ to be irreducible w.r.t. to a non-singular measure. We therefore exclude $0$ from the state space that is now equal to $\ZZ= \R^{\dim}_{\neq}$.

\todo{DO we want that: Remark that it amounts to consider the function $\tilde{f}$ such that $\tilde{f}(\xstar)=+\infty$ instead of $f$. Both $f$ and $\tilde{f}$ are equal almost everywhere and have the same essential infimum. Take initialization $\Z_{0}$ in $\ZZ$.}

\section{Study of the Normalized Chain}\label{sec:stability}

We study in this section different properties of the homogeneous Markov chain $\Z=(\Zt)_{t \in \NNN}$ defined in Section~\ref{sec:defZ}. Those properties will imply linear convergence of the $(1+1)$-ES as we will see in Section~\ref{sec:linear-convergence}. We start in the next section by expressing the transition kernel of the Markov chain.

\subsection{Transition Probability Kernel}

We follow standard notations and terminology for a time homogeneous Markov chain $(\Zt)_{t \in \NNN}$ on a topological space $\ZZ$. The Borel sets of $\ZZ$ are denoted $\B(\ZZ)$. A kernel $T$ is any function on $\ZZ \times \B(\ZZ)$ such that $T(.,A)$ is measurable for all $A \in \B(\ZZ)$ and $T(\z,.)$ is a measure for all $\z \in \ZZ$. The transition probability kernel for $(\Zt)_{t \in \NNN}$ is a kernel $P$ such that $P(.,A)$ is a non-negative measurable function for all $A \in \B(\ZZ)$ and the measure $P(\z,.)$ for all $\z$ is a probability measure. It is defined as
$$
P(\z,A) = P_{\z}(\Z_{1} \in A) \enspace,
$$
where $P_{\z}$ denotes the probability law of the chain under the initial condition $\Z_{0} = \z$. Similarly $E_{\z}$ denotes the expectation of the chain under the initial condition $\Z_{0} = \z$. If a probability $\mu$ on $(\ZZ,\B(\ZZ))$ is the initial distribution, the probability law and expectation under $\mu$ are denoted $P_{\mu}$ and $E_{\mu}$.
The \emph{n-step transition probability law} is defined iteratively by setting $P^{0}(\z,A)=\delta_{\z}(A)$ and for $t \geq 1$, inductively by
$$
P^{t}(\z,A)=\int P(\z,d\y)P^{t-1}(\y,A) \enspace.
$$
The relation $P^{t}(\z,A) = P_{\z}( \Z_{t} \in A ) $ holds. With an abuse of notations similar to \cite[p~56]{Tweedie:book1993}, we will also for instance denote $\Pr(\N \in A) $ or $\Pr(f(\z+\N) > f(\z))$ for the probability of the events $\{ \N \in A \}$, $\{ f(\z+\N) > f(\z) \}$ (where $\N$ will typically be a standard normal multivariate distribution) without specifically defining the space where $\N$ exists which could be the space where $\Z$ is defined or another space. Similarly $E[1_{\{ f(\z+\N) > f(\z) \}}]$ will be used for the expectation of the random variable $1_{\{ f(\z+\N) > f(\z) \}}$.

We derive in the next proposition an expression for the transition kernel of $(\Zt)_{t \in \NNN}$ when $f$ is a scaling-invariant function.

\begin{proposition}\label{prop:kernelplus}
Let $f:\R^{n} \mapsto \R$ be a scaling-invariant function and let $\Z$ be the Markov chain defined in \eqref{eq:transitionZ}. Its transition probability kernel is given for all  $\z \in \ZZ= \R^{\dim}_{\neq}$ and $A \in \B(\ZZ)$ by
\begin{align}
P(\z,A)& = \int 1_{A}(\uu) \qq(\z,\uu)  d \uu
+ 1_{A}(\z \factonefifth^{1/q}) \Pr(f(\z+\N) > f(\z) )
\end{align}
where $\qq(\z,\uu)=1_{\{f(\uu) \leq f(\z/\factonefifth) \}}(\uu) \pp(\factonefifth \uu - \z) \factonefifth $ with $\pp$ the density of a standard multivariate normal distribution and $\N \sim \Normal$.
\end{proposition}

{\em Proof.}
Given $\Z_{0} = \z$, and $\U_{0}^{1} = \N$ where $\N \sim \N(0,\Id)$, $\Z_{1}$ satisfies
$
\Z_{1} = \frac{\z +  \N 1_{\{ f( \z + \N) \leq f(\z) \}}}{(\factonefifth - \factonefifth^{-1/q})1_{\{f(\z+ \N) \leq f(\z) \}} + \factonefifth^{-1/q}} 
$.
Hence, the transition probability kernel satisfies $P(\z,A)=$
$$
\Pr \left( \frac{\z +  \N}{\factonefifth} \in A \,\, \cap \{f(\z+ \N) \leq f(\z)  \} \right) + \Pr \left( \frac{\z}{\gamma^{-1/q}} \in A \cap \{f(\z+ \N) > f(\z)  \} \right) \enspace,
$$
and thus satisfies
\begin{align*}
P(\z,A)&= \int 1_{A} \left( \frac{\z + \uu}{\factonefifth} \right) 1_{\{f(\z+\uu) \leq f(\z) \}}(\uu) \pp(\uu) d\uu
+ 1_{A}(\z \factonefifth ^{1/q})  \Pr(f(\z+\N) > f(\z) )\\
& = \int 1_{A}(\bar \uu) \underbrace{1_{\{f(\factonefifth \bar \uu) \leq f(\z) \}}(\bar \uu) \pp(\factonefifth \bar \uu - \z) \factonefifth}_{\qq(\z,\bar \uu)} d \bar \uu
+ 1_{A}(\z \factonefifth ^{\frac1q}) \Pr(f(\z+\N) > f(\z) ). \endproof
\end{align*}

\mathnote{I take out this part for the submission as length start to be an issue and I do not need the Feller property for the paper. Note that for proving the weak Feller property (would need to be checked carrefully) the cited proposition does not apply directly. I however need to do an ad-hoc proof, showing that the points of discontinuity in $\n$ are not problematic to apply the dominated function theorem.}
\del{The transition probability kernel $P$ acts on bounded functions through the mapping
$$
Ph(\z) = \int P(\z,d \y) h(\y), \z \in \ZZ \enspace.
$$
For locally compact separable metric topological space $\ZZ$, the class of bounded continuous functions from $\ZZ$ to $\R$ is denoted  $\mathcal{C}(\ZZ)$. A chain $\Z$ is said weak Feller if $P$ maps  $\mathcal{C}(\ZZ)$ into $\mathcal{C}(\ZZ)$ and strong Feller if it maps bounded functions to  $\mathcal{C}(\ZZ)$. The Markov chain associated to the  \plusES-ES is weak Feller. The proof of this result follows the exact same lines as the proof of \cite[Proposition 6.1.3.]{Tweedie:book1993}. We formalize this result needed later in the following lemma:
\begin{lemma}\label{lem:weakFeller}
The Markov chain associated to the \plusES-ES is weak Feller.
\end{lemma}
\begin{proof}
\nnote{Actually the proof cited before is using stronger assumptions that are not needed for the Dominated convergence theorem. Hence do a quick and correct proof}
\nnote{Probably state the lemma in terms of property of the mapping $\GG$ $\G$}
\end{proof}

However the singular part of the transition kernel \plusES makes that the chain is not strong Feller.
\todo{actually it's not just the singular part that makes it not strong feller - for the comma case it is also not strong feller}
}

\subsection{Irreducibility, Small Sets and Aperiodicity}\label{sec:ISSA}

A Markov Chain $\Z=(\Zt)_{\t \in \NNN}$ on a state space $\ZZ$ is said $\varphi$-irreducible if there exists a measure $\varphi$ on $\ZZ$ such that for all $A \in \B(\ZZ)$, $\varphi(A) > 0$ implies that $P_{\z}(\tau_{A} < \infty) > 0$ for all $\z \in \ZZ$ where $\tau_{A}=\min \{t > 0 : \Zt \in A\}$ is the first return time to $A$. Another equivalent definition is for all $\z \in \ZZ$ and for all $A \in \B(\ZZ)$
$$
\varphi(A) > 0 \Rightarrow \exists  \, \, t_{\z,A} \in \NNN \text{ such that} \,\, P_{\z}(\Z_{ t_{\z,A}} \in A) > 0 \enspace.
$$
Given that a chain $\Z$ is $\varphi$-irreducible, there exists a maximal irreducibility measure $\psi$ and all maximal irreducibility measure are equivalent (see \cite[Proposition~4.4.2]{Tweedie:book1993}). 
The set of positive $\psi$-measure is denoted
$$
\B^{+}(\ZZ) := \{ A \in \B(\ZZ) : \psi(A) > 0 \} \enspace.
$$
In the sequel we continue to denote $\psi$ the maximal irreducibility measure and hence if $\Z$ is $\psi$-irreducible it means that it is $\varphi$-irreducible for some $\varphi$ and that $\psi$ is a maximal irreducibility measure.
A set $A$ is \emph{full} if $\psi(A^{c}) = 0 $ and \emph{absorbing} if $P(\z,A) =1$ for $\z \in A$.
In addition, a set $C$ is a \emph{small set} if there exists $t \in \NNN$ and a non-trivial measure $\nu_{t}$ on $\B(\ZZ)$ such that for all $\z \in C$
\begin{equation}
P^{t}(\z,A) \geq \nu_{t}(A) \, , A \in \B(\ZZ) \enspace.
\end{equation}
The small set is then called a $\nu_{t}$-small set. 
Consider a small set $C$ satisfying the previous equation with $\nu_{t}(C) >0$ and denote $\nu_{t} = \nu$. The chain is called aperiodic if the g.c.d.\ of the set
$$
E_{C} = \{ k \geq 1: C \text{ is a } \nu_{k} \mbox{-small set with } \nu_{k} = \alpha_{k} \nu \mbox{ for some } \alpha_{k} >0 \}
$$
is one for some (and then for every) small set $C$.


We establish  now the $\varphi$-irreducibility, identify some small sets and show the aperiodicity of the normalized chain associated to the \plusES-ES. 

\subsubsection{$\varphi$-irreducibility}
We denote $\muLeb$ the Lebesgue measure on $\ZZ = \R^{\dim}_{\neq}$. 
We prove in the next proposition that the normalized MC associated to the \plusES-ES is irreducible with respect to the Lebesgue measure.

\begin{proposition}Assume that $f$ satisfies Assumptions~\ref{ass:f} and  is continuous on $\Rnstar$. Assume that $\factonefifth > 1$. Then, the Markov chain $\Z$ associated to the \plusES-ES is irreducible w.r.t. the Lebesgue measure $\mu_{\rm Leb}$.
\end{proposition}
\begin{proof}
Let $\z \in \ZZ$ and $\LLL_{f(\z/\gamma)}^{f}$ be the sublevel set $\{ \uu \in \ZZ, f(\uu) \leq f(\z/\gamma) \} = \{ \uu \in \ZZ, f( \gamma \uu) \leq f(\z) \}$. And let $A \in \B(\ZZ)$ such that $\mu_{\rm Leb}(A) > 0$. By the regularity of the Lebesgue measure we can include a compact $K$ in A such that $\mu(K) > 0$. Since $K \subset A$, for all $\z$, $P(\z,A) \geq P(\z,K).$
\mathnote{ Regularity of a measure would be with a closed and not a compact (see Bilingsley) - however for Lebesgue measure, works with compact - note also that in Billingsley Theorem 12.3 is stated for $A$ such that $\mu(A) < \infty$, however also holds for $A$ with infinite measure. For a nice proof see the measure-note.pdf Theorem 2.23 - class notes by John K. Hunter.}

If (i) $K \subset \LLL_{f(\z/\gamma)}^{f}$ then $ P(\z,K) = \int_{K}   \factonefifth \pp(\factonefifth \uu - \z) d\uu >0$ as $\uu \in \R^{\dim} \mapsto \pp(\factonefifth \uu - \z) > 0 $. 

\mathnote{Here we do not need that $K$ is compact since for a positive function $f$, $\int f = 0$ if and only if $f = 0$ a.e. Just take $f(\uu) = 1_{K}(\uu) \pp( \gamma \uu - \z)$. We find that for any lebesgue measurable set $A$,  $\mu(A) = 0 $ iff $\int_{A} \gamma \pp( \gamma \uu - \z) =0 $. Hence $\mu(A) > 0$ iff $\int_{A} \gamma \pp( \gamma \uu - \z) > 0  $. Of course if we have a compact we can just use the fact that $\uu \in \R^{\dim} \mapsto \pp(\factonefifth \uu - \z) > 0 $ is continuous on the compact $K$ and hence bounded and reaches its bound, hence there exists $B$ such that $\pp(\factonefifth \uu - \z) \geq B $ and then $P(\z,K) \geq \gamma B$.}

If (ii) $K$ is not included in $ \LLL_{f(\z/\gamma)}^{f} $, then by sampling $\N$ such that $f(\z+\N) > f(\z)$ (which happens with strictly positive probability since by Lemma~\ref{lem:propf}, $ \LLL_{f(\z)}^{f} $ is bounded, hence sampling such that $f(\z+\N) > f(\z)$ can be achieved by sampling outside a ball), $\Z_{1}=\z \factonefifth^{1/q}$ which is at a  larger distance from $0$ (as we assumed that $\gamma > 1$). By repeating this, we build a sequence $\Zt= \z \factonefifth^{\t/q}$ and $f(\Zt) = ({\factonefifth^{\t/q}})^{\alpha} f(\z)$ hence $f(\Zt)$ and $f(\Zt / \gamma)$ go to $\infty$. The set $K$ being compact we can find a ball $\BB(0,\rho)$ such that $K \subset \BB(0,\rho)$. In addition from Lemma~\ref{lem:si}, we know that for all $\rho$, $\BB(0,\rho) \subset \LLL_{\rho^{\alpha}M}$, hence choosing $t$ large enough such that $\LLL_{\rho^{\alpha} M}^{f} \subset \LLL_{f(\Zt/\gamma)}^{f}$, we have that $K \subset \LLL_{f(\Zt/\gamma)}^{f}$ and by (i) $P(\Zt=\z\gamma^{t/q},K) > 0$. 
\newcommand{\zbar}{\bar{\z}}
Thus $P^{t+1}(\z,A) \geq P^{t+1}(\z,K) \geq \Pr(f(\z + \N) > f(\z)) \ldots \Pr\left(f(\z \factonefifth^{\frac{t-1}{q}}+\N) > f(\z \factonefifth^{\frac{t-1}{q}}) \right) P(\z\gamma^{t/q},K) \geq \theta^{t}  P(\z\gamma^{t/q},K)> 0$ where $\theta > 0$ is a lower bound on $\zbar \to \Pr(f(\zbar + \N) > f(\zbar)) $ on a ball that includes $ \LLL_{f(\Zt)}^{f} $. 
Indeed, since $f(\Z_{k})$ increases, for all $k \leq t$, $f(\Z_{k}) \in \LLL_{f(\Zt)}^{f}$.  However according to Lemma~\ref{lem:si}, there exists $R$ such that $ \LLL_{f(\Zt)}^{f} \subset B(0,R) $. Then $ \{ \N \in B(0,2R)^{c} \}  \subset \{ f(\zbar + \N) > f(\zbar) \}$ for all $\zbar$ in $\LLL_{f(\Zt)}^{f} $. We can take $\theta = \Pr( \N \in B(0,2R)^{c} )$.
\end{proof}

\subsubsection{Small Sets and Aperiodicity}\label{small-sets}
We investigate small sets for the \plusES-ES assuming that $f$ is positively homogeneous with degree $\alpha$ with $f(\x) >0$ for $\x \neq 0$ and $f$ is continuous on $\Rnstar$. Consider sets $D_{[l_{1},l_{2}]}$ with $ 0 < l_{1} < l_{2} $ defined as
\newcommand{\DD}{D_{[l_{1},l_{2}]}}
\begin{equation}\label{eq:yang}
D_{[l_{1},l_{2}]} : = \{ \z \in \ZZ, l_{1} \leq f(\z) \leq l_{2} \} \enspace.
\end{equation}
Because $f$ is continuous, the sets $\DD = f^{-1}([l_{1},l_{2}])$ are closed and by Lemma~\ref{lem:si} they are also bounded such that the sets $\DD$ are compact sets. We prove in this section that the sets $\DD$ are small sets for the Markov chain $\Z$.
\begin{lemma}\label{lem:smallset}
Assume that $f$ is positively homogeneous with degree $\alpha$, $f(\x)>0$ for $\x \neq 0$ and $f$ is continuous on $\Rnstar$. Assume that $\gamma > 1$. Let $\DD$ be a set of the type \eqref{eq:yang} with $0<l_{1} < l_{2}$.
Let $t_{0} \geq 1$ and let $R > 0$ such that $ \LLL_{\gamma^{\frac{\alpha t_{0}}{q}} l_{2}} \subset B(0,R)$ (see Lemma~\ref{lem:si}). Then for all $\z $ in $\DD$ and for all $t \leq t_{0}$
\begin{equation}\label{eq:maj-proba}
\Pr \left( f( \z \gamma^{\frac{t}{q}}  + \N  ) > f \left(\z \gamma^{\frac{t}{q}} \right) \right) \geq 
\Pr \left( \N \in B(0,2R)^{c} \right) =:\theta_{\ref{lem:smallset}}
\end{equation}
where $\N \sim \N(0,\Id)$.
For all $\z \in \DD$ and $t \leq t_{0}$, the following minorization holds, for $A \in \B(\ZZ)$
\begin{equation}\label{eq:ntplusone-smallset}
P^{t+1}(\z,A) \geq\theta_{\ref{lem:smallset}}^{t} \gamma \delta_{t} \int 1_{A \cap \DD}(\uu) 1_{ \{  f(\gamma \uu) \leq \gamma^{{t \alpha}/{q}} l_{1} \} }(\uu)  d \uu   =  : \nu_{t+1}(A) \enspace,
\end{equation}
where $\delta_{t} = \min_{(\z, \uu) \in D_{[l_{1},l_{2}]}^{2}} \pp( \gamma \uu - \z \gamma^{t/q} ) > 0$.
In addition $\nu_{t+1}$ is a non-trivial measure if $t > q$ and hence $\DD$ is a $\nu_{t+1}$-small set provided $t > q$.
\end{lemma}
\mathnote{$ \min_{\z \in D_{[l_{1},l_{2}]}} \pp( \gamma \uu - \z \gamma^{t/4} )$ 
 a measurable application and even a continuous application if we assume $p$ continuous. I guess I do not have the simplest proof for that. I basically use the uniform continuity of a continuous function on a compact. Consider $u_{n}$ converging to $u_{0}$ where we want to prove the continuity of $g(u) = \min_{\z \in  D_{[l_{1},l_{2}]} }  \pp( \gamma \uu - \z \gamma^{t/4} ) = \min_{\z \in  D_{[l_{1},l_{2}]} }  \pp(\z,u)$ Let $z_{0}$ such that $g(u_{0}) = \pp( \z_{0}, u_{0}) $ (exists because we have a continuous function on a compact. Let us denote $K=D_{[l_{1},l_{2}]}$. On $K \times B(u0,\epsilon)$, $\pp(\z,u)$ is uniformly continuous. Consider $\min_{K \times B(u0,\epsilon)}  \pp(\z,u):=A$. Then $A < g(u_{0})$. We need to show that for $u_{n}$ close from $u_{0}$, $z_{n}$ where the minimum is reached cannot be far from $\z_{0}$ which implies by uniform continuity that $g(un)$ is close from $g(u_{0})$. We proceed by proving that both (un,zn) [associated to $g(un)$] and (u0,z0) [associated to $g(u0)$] are close from the point where $A$ is reached. by contradiction if it was not close then we would find in a neighborhood the point where the global minimum $A$ is reached a point $u_{0}, zz$ that would satisfy $\pp(zz,uo) < g(u0)$ which is impossible. Same for (un,zn). Overall both (un,zn) and (u0,z0) are closed from the point where $A$ is reached and therefore they are close together.
 Note: 3 August 2013: I do not see that the previous proof is correct ...
 What is going on: 1) given $u_{0}$ via the compact we know that the minimum is reached at a point $z_{0}$. 2) Moving $u_{0}$ in a ball has the effect to move also via the translation $\gamma u - \z \gamma^{t/4}$, $z_{0}$ in a ball as well. Basically the translation moves every point of the compact continuously. Then we check the density $p$ for this moving ``z-domain''.  In case the minimum was only reached at one point of the border of the compact (BTW, the minimum with a gaussian distribution can only be reached at the border), then the minimum will stay reached at the same $z_{0}$. If it was reached at two different points of the border, then the effect of moving in the ball will be to determine one of the two minimum - still continuity remains.

Would be good to have a clean proof: check slide - might give the proof idea.

\url{http://www.math.ku.dk/~erhansen/Markov_04/}

Usefull 
\url{http://citeseerx.ist.psu.edu/viewdoc/download?doi=10.1.1.53.3081&rep=rep1&type=pdf}

Meyn Tweedie page 89 (or Section 4.3.3) - require for additive RW to have the density of the increment bounded away from zero in the neighborhood of zero.
 }
 
\mathnote{Pour un autoregressive process $\Ztt = \Zt + Wt$ where $Wt$ admits a density $\gamma(x)$ such that for all compact C $\inf_{x \in C} \gamma(x) >0$. The compacts are small sets. (see cours Hansen \url{http://www.math.ku.dk/~erhansen/Markov_04/}).
The slides where it is stated are: CoursHansenMCMC-slides-AR1-smallsets.ps and the solution is in the handwritten notes: HansenHandOut.pdf. (p 13)
It turns out that this is the same argument that we find in the slides of Rosenthal (p 13):
density of the kernel $f(x,y)$, satisfy $f(x,y) > \delta$ for all $x, y \in C$ (true if C compact and f satisfy $\inf_{x,y \in C} f(x,y) > 0$
$ P(x,A) \geq \delta Leb(A \cap C) = \delta Leb(C)[ Leb(A \cap C)/Leb(C)]$ This latter equation we have exhibited the probability measure. 
Note we need that the sum of two compacts is still a compact (see the proof in the additional material, exo 7 if needed).
}

\begin{proof}
Note first that for all $t \leq t_{0}$, for all $\z \in \DD$, $\LLL_{f(\z \gamma^{t/q})} \subset \LLL_{\gamma^{{\alpha t_{0}}/{q}} l_{2}} $ (we use here that $\gamma > 1$). We now claim that if $\N \in B(0,2R)^{c}$, then $ f( \z \gamma^{{t}/{q}}  + \N  ) > f \left(\z \gamma^{{t}/{q}} \right)$. Indeed, if $\N \in B(0,2R)^{c}$, then for all $\z$ in $\DD$, $\z \gamma^{\frac{t}{q}}  $ is inside $\LLL_{f(\z \gamma^{t/q})} \subset\LLL_{\gamma^{{\alpha t_{0}}/{q}} l_{2}} \subset B(0,R)$ (by definition of $R$) and hence $\z \gamma^{\frac{t}{q}}  + \N$ will be outside $B(0,R)$, i.e., outside $\LLL_{f \left(\z \gamma^{{t}/{q}} \right)}$, i.e., $f( \z \gamma^{{t}/{q}}  + \N  ) > f \left(\z \gamma^{{t}/{q}} \right)$. Hence \eqref{eq:maj-proba}. \todo{explain that it does not work for strong aperiodicity - explain the intuitive idea about the construction}We will now prove \eqref{eq:ntplusone-smallset}. We lower bound the probability $P^{t+1}(\z,A)=P_{\z}( \Z_{t+1} \in A) $ by the probability to reach $A$ in $t+1$ steps starting from $\z$ by having no success for the first $t-1$ iterations:
\begin{align*}
P^{t+1}(\z,A) \geq \underbrace{P_{\z} \left( \{ \Z_{t+1} \in A \} \cap \{ f( \Z_{t} + \Ut^{1}) \leq f( \Z_{t}) \} \bigcap_{k=0}^{t-1} \{ f( \Z_{k} + \U_{k}^{1}) > f(\Z_{k})  \} \right)}_{A_{1}} \enspace.
\end{align*}
However if $ \{ f( \Z_{k} + \U_{k}^{1}) > f(\Z_{k})  \}$ then $\Z_{k+1} = \Z_{k} \gamma^{1/q} $ such that given that $\Z_{0} = \z$, the following equalities between events holds 
\begin{multline*}
\{ \Z_{t+1} \in A \} \cap \{ f( \Z_{t} + \U_{t}^{1}) \leq f( \Z_{t} ) \}
\bigcap_{k=0}^{t-1} \{ f( \Z_{k} + \U_{k}^{1}) > f(\Z_{k})  \}  = \\
\left\{ \frac{\z \gamma^{t/q} + \U_{t}^{1}}{\gamma} \in A  \right\} 
\cap \left\{ f( \z \gamma^{t/q} + \U_{t}^{1}) \leq f( \z \gamma^{t/q} ) \} 
\bigcap_{k=0}^{t-1}  \{ f( \z \gamma^{k/q} + \U_{k}^{1}) > f(\z \gamma^{k/q})  \right\} \enspace.
\end{multline*}
Hence by independence of the $(\U_{k})_{0 \leq k \leq t}$, $A_{1} =$
\begin{align*}
& 
P_{\z} \left(  \frac{\z \gamma^{t/q} + \U_{t}^{1}}{\gamma} \in A  \cap \{ f( \z \gamma^{\frac{t}{q}} + \U_{t}^{1}) \leq f(\z \gamma^{\frac{t}{q}} ) \} \right) \prod_{k=0}^{t-1} P_{\z} \left( f( \z \gamma^{\frac{k}{q}} + \U_{k}^{1}) > f(\z \gamma^{\frac{k}{q}}) \right)
\intertext{using now \eqref{eq:maj-proba}}
& \geq P_{\z} \left(  \frac{\z \gamma^{t/q} + \U_{t}^{1}}{\gamma} \in A \cap \{ f( \z \gamma^{t/q} + \U_{t}^{1}) \leq f( \z \gamma^{t/q} ) \}  \right) \theta_{\ref{lem:smallset}}^{t} \\
& = \theta_{\ref{lem:smallset}}^{t} \int 1_{A} \left( \frac{\z \gamma^{t/q} + \uu}{\gamma}  \right) 1_{\{ f( \z \gamma^{t/q} + \uu) \leq  f( \z \gamma^{t/q} ) \}}(\uu) \pp(\uu) d \uu \\
& =  \theta_{\ref{lem:smallset}}^{t} \int 1_{A} \left( \bar \uu \right) 1_{\{ f( \gamma \bar \uu) \leq  f( \z \gamma^{t/q} ) \}}(\bar \uu) \pp( \gamma \bar \uu - \z \gamma^{t/q}) \gamma d \bar \uu \\
& \geq \theta_{\ref{lem:smallset}}^{t} \int 1_{A \cap \DD} \left( \bar \uu \right) 1_{\{ f( \gamma \bar \uu) \leq  f( \z \gamma^{t/q} ) \}}(\bar \uu) \pp( \gamma \bar \uu - \z \gamma^{t/q}) \gamma d \bar \uu  \enspace .
\end{align*}
For all $\z \in \DD$, $\bar \uu$, $ 1_{\{ f( \gamma \bar \uu) \leq  f( \z \gamma^{t/q} ) \}}(\bar \uu) =  1_{\{ f( \gamma \bar \uu) \leq  \gamma^{{t\alpha}{q}}  f( \z ) \}}(\bar \uu) \geq 1_{\left\{ f( \gamma \bar \uu) \leq  \gamma^{{t\alpha}{q}} l_{1} \right\}}(\bar \uu) $ and for all $(\z,\bar \uu) \in \DD$, $\pp( \gamma \bar \uu - \z \gamma^{t/q}) \geq  \min_{(\z,\bar \uu) \in \DD^{2}} \pp( \gamma \bar \uu - \z \gamma^{t/q} )$. Since $\DD$ is compact, $\{\gamma \bar \uu - \z \gamma^{t/q},  (\z,\bar \uu) \in \DD^{2} \}$ is also compact and thus there exists $\delta_{t}$ such that $\min_{(\z,\uu) \in \DD^{2}} \pp( \gamma \bar \uu - \z \gamma^{t/q} ) \geq \delta_{t} > 0$.
\mathnote{Continuous image of a compact is a compact.}
\mathnote{that's here that we need that the sum of two compacts is compact}
\mathnote{actually two proofs of the same fact - OK}
Hence
$$
P^{t+1}(\z,A) \geq A_{1} \geq  \theta_{\ref{lem:smallset}}^{t} \delta_{t} \gamma \int 1_{A \cap \DD}(\uu)1_{\left\{ f( \gamma \uu) \leq  \gamma^{{t\alpha}/{q}} l_{1} \right\}}(\uu)  d\uu
$$
which is a non-trivial measure if $t > q$.
\mathnote{$  l_{1} \leq f(\uu) $ and $\gamma^{\alpha} f(\uu) \leq \gamma^{t \alpha/q} l_{1} $.}
\end{proof}

Remark that the constant $\theta_{\ref{lem:smallset}}$ defined in \eqref{eq:maj-proba} and used in \eqref{eq:ntplusone-smallset} depends on $t_{0}$ as the radius $R$ of the ball where the sublevel set $ \LLL_{\gamma^{{\alpha t_{0}}/{q}} l_{2}}$ is included depends on $t_{0}$. To prove the aperiodicity of the chain we construct a joint minorization measure $\nu$ working for two consecutive integers having hence $1$ as greatest common divisor and that satisfies $\nu(\DD) > 0$.  More precisely we prove the following proposition.
\newcommand{\qqq}{\bar{q}}
\begin{proposition}\label{prop:D-6-7smallSets}
Assume  that $f$ is positively homogeneous with degree $\alpha$, $f(\x)>0$ for $\x \neq 0$ and $f$ is continuous on $\Rnstar$.  Assume that $\gamma > 1$. Let $\DD$ be a set of the type \eqref{eq:yang}. Let $\bar{q} = \lfloor q \rfloor +1 $.  Then for all $\z \in \DD$, $A \in \B(\ZZ)$
\begin{align}\label{hom1}
P^{\bar{q}+1}(\z, A) & \geq  \zeta^{\qqq}  \nu(A)   \\\label{hom2}
P^{\bar{q}+2}(\z, A) & \geq \zeta^{\qqq+1}  \nu(A) 
\end{align}
where $\nu$ is the measure defined by
$\nu(A) = \delta \gamma \int 1_{A \cap \DD} ( \uu)1_{\left\{ f( \gamma \uu) \leq  \gamma^{\qqq \alpha/q} l_{1} \right\}}(\uu)  d \uu$ with $\delta = \min \{ \delta_{\qqq}, \delta_{\qqq+1} \}$ and $\zeta$ is the constant $\theta_{\ref{lem:smallset}}$ in \eqref{eq:maj-proba} for $t_{0}=\qqq+2$.
In addition $\nu(\DD) > 0$ which implies that the chain $\Z$ is aperiodic.
\end{proposition}
\begin{proof}
From Lemma~\ref{lem:smallset}
\begin{align}\label{eq:tchoug}
P^{\qqq+1}(\z,A) & \geq \zeta^{\qqq} \delta_{\qqq} \gamma \int 1_{A \cap \DD}(\uu) 1_{ \{  f(\gamma \uu) \leq \gamma^{ {\qqq \alpha}/{q}} l_{1} \} }(\uu)  d \uu  \enspace, \\ \label{eq:tchoug2}
P^{\qqq+2}(\z,A) & \geq \zeta^{\qqq+1} \gamma \delta_{\qqq+1}  \int 1_{A \cap \DD}(\uu) 1_{ \{  f(\gamma \uu) \leq \gamma^{{(\qqq+1) \alpha}/{q}} l_{1} \} }(\uu) d \uu  \enspace.
\end{align}
Since $\delta \leq \delta_{\qqq} $ we find using \eqref{eq:tchoug} that
$$
P^{\qqq+1}(\z,A) \geq \zeta^{\qqq} \delta \gamma  \int 1_{A \cap \DD}(\uu) 1_{ \{  f(\gamma \uu) \leq \gamma^{{\qqq \alpha}/{q}} l_{1} \} }(\uu) d \uu,   
$$
which is exactly \eqref{hom1}. Using in addition the fact that $1_{ \{  f(\gamma \uu) \leq \gamma^{{(\qqq+1) \alpha}/{q}} l_{1} \} } (\uu) \geq 1_{ \{  f(\gamma \uu) \leq \gamma^{{\qqq \alpha}/{q}} l_{1} \} }(\uu)  $ that we  inject in \eqref{eq:tchoug2}, we find \eqref{hom2}.
Since $\qqq/q > 1$, $\nu(\DD) > 0$. As two consecutive integers have $1$ as g.c.d., the g.c.d.\ of $\qqq+1$ and $\qqq+2$ is one and hence the chain $\Z$ is aperiodic.
\mathnote{proof of two consecutive integer have 1 has greatest common divisor can be found on wikipedia: \url{http://www.proofwiki.org/wiki/Consecutive_Integers_are_Coprime}}
\end{proof}

\subsection{Geometric Ergodicity}

In this section we derive the geometric ergodicity of the chain $\Z$. Geometric ergodicity will imply the other stability properties needed, namely positivity and Harris recurrence whose definitions are reminded below. First, let us recall that
a $\sigma$-finite measure $\pi$ on $\B(\ZZ)$ with the property
$$
\pi(A) = \int_{\ZZ} \pi(d \z) P(\z,A),  \, A \in \B(\ZZ)
$$
is called invariant. A $\varphi$-irreducible chain admitting an invariant probability measure is called a \emph{positive} chain.
Harris recurrence is a concept ensuring that a chain visits the state space sufficiently often. It is defined for a $\psi$-irreducible chain as: A $\psi$-irreducible Markov chain is \emph{Harris-recurrent} if for all $A \subset \ZZ$ with $\psi(A) > 0$, and for all $\z \in \ZZ$, the chain will eventually reach $A$ with probability $1$ starting from $\z$, formally if $P_{\z}(\eta_{A} = \infty) = 1$ where $\eta_{A}$ be the \emph{occupation time} of $A$, i.e., $\eta_{A} = \sum_{t=1}^{\infty} 1_{\Zt \in A}$. An (Harris-)recurrent chain admits an unique (up to a constant multiples) invariant measure \cite[Theorem~10.0.4]{Tweedie:book1993}.

For a function $V \geq 1$, the $V$-norm for a signed measure $\nu$ is defined as
$$
\| \nu \|_{V} = \sup_{k: |k| \leq V} | \nu(k)| = \sup_{k: |k| \leq V} | \int k(\y) \nu( d \y)| \enspace.
$$
Geometric ergodicity translates the fact that convergence to the invariant measure takes place at a geometric rate. Different notions of geometric ergodicity do exist (see \cite{Tweedie:book1993}) and we will consider the form that appears in the following theorem. For any $V$, $PV$ is defined as $PV(\z) := \int P(\z, d\y) V(\y)$.
\begin{theorem}
(Geometric Ergodic Theorem \cite[Theorem 15.0.1]{Tweedie:book1993}) 
Suppose that the chain $\Z$ is $\psi$-irreducible and aperiodic. Then the following three conditions are equivalent:
(i) The chain $\Z$ is positive recurrent with invariant probability measure $\pi$, and there exists some petite set $C \in \B^{+}(\ZZ) $, $\rho_{C} < 1$ and $M_{C} < \infty$ and $P^{\infty}(C) > 0$ such that for all $\z \in C$
$$
|P^{t}(\z,C) - P^{\infty}(C) | \leq M_{C} \rho_{C}^{t}.
$$
(ii) There exists some petite set $C$ and $\kappa > 1$ such that
$$
\sup_{\z \in C} E_{\z}[\kappa^{\tau_{C}}] < \infty \enspace.
$$
(iii) There exists a petite set $C \in \B(\ZZ)$, constants $b < \infty$, $\vartheta < 1$ and a function $V \geq 1$ finite at some one $\z_{0} \in \ZZ$ satisfying
\begin{equation}\label{eq:driftgerd}
PV(\z) \leq \vartheta V(\z) + b 1_{C}(\z),  \z \in \ZZ.
\end{equation}
Any of these three conditions imply that the following two statements hold. The set $S_{V}=\{ \z: V(\z) < \infty\}$ is absorbing and full, where $V$ is any solution to \eqref{eq:driftgerd}. Furthermore, there exist constants $r>1$, $R < \infty$ such that for any $\z \in S_{V}$
\begin{equation}\label{eq:C-Vnorm}
\sum_{t} r^{t} \| P^{t}(\z,.) - \pi \|_{V} \leq R V(\z) \enspace.
\end{equation}
\end{theorem}
The drift operator is defined as 
$
\Delta V(\z) = PV(\z) - V(\z)
$.
The inequality \eqref{eq:driftgerd} is called a drift condition that can be re-written as
$$
\Delta V(\z) \leq \underbrace{(\vartheta - 1)}_{<0} V(\z) + b 1_{C}(\z) \enspace.
$$
$P$ is then said to admit a drift towards the set $C$.
The previous theorem is using the notion of petite sets but small sets are actually also petite sets (see Section~5.5.2 \cite{Tweedie:book1993}). We will in the sequel prove a geometric drift towards a small set $C$ that will hence imply a geometric drift towards petite set. It will subsequently imply the existence of a probability invariant measure and Harris recurrence 
\cite{Tweedie:book1993}. 

\subsubsection{Geometric Drift Condition for Positively Homogenous Functions}

In this section we investigate drift conditions for functions that are a monotonically increasing transformation of a positively homogeneous function, i.e., $h = g \circ f$ for $f$ a positively homogeneous function with degree $\alpha$ and $g \in \Monotone$. 

\newcommand{\zzz}{\tilde{\z}}

We have shown that the sets $D_{[l_{1},l_{2}]}$ are some small sets for $\Z$ (under the assumptions of Lemma~\ref{lem:smallset}). Hence proving negativity of the drift function outside a small set requires to prove negativity for $f(\z)$ ``large'' as well as for $f(\z)$ close to $0$.
We are going to prove that under some regularity assumptions on $f$, the function
\begin{equation*}
V(\z)= f(\z) 1_{\{ f(\z) \geq  1 \}} + \frac{1}{f(\z)} 1_{\{ f(\z) <  1 \}}
\end{equation*}
satisfies a geometric drift condition for the \plusES-ES algorithm provided $\factonefifth > 1$ and the expected inverse of the step-size change to the $\alpha$ on \emph{linear functions} is strictly smaller one that directly translates into:
$$
\frac12 \left( \frac{1}{\factonefifth^{\alpha}} + \factonefifth^{\alpha/q} \right) < 1 \enspace.
$$
\anne{Note that I need to verify the integrability against the drift function. I am wondering whether Meyn-Tweedie talk about that. Done in a technical result later on.}
\mathnote{I guess their trick is that everything can be possibly infinite.}

Given the shape of the small sets proven in Section~\ref{small-sets}, to establish a geometric drift condition,\smalltodo{not equivalent - think about lim sup} it is enough to prove that the limit of $PV/V$ is strictly smaller 1 when $\z$ goes to $0$ and to $\infty$:

\begin{lemma}\label{lem:CSdrift}
Let $f$ be positively homogeneous function with degree $\alpha$ and $f(\x) > 0$ for $\x \neq 0$ and $f$ continuous on $\Rnstar$. Assume that $\gamma > 1$. Let $V$ be a function finite at some one $\z_{0} \in \ZZ$ and such that $V \geq 1$
that satisfies
\begin{equation}\label{eq:toub}
\lim_{\| \z \| \to \infty} \frac{PV(\z)}{V(\z)} < 1 \mbox{ and } \lim_{\|\z\| \to 0} \frac{PV(\z)}{V(\z)} < 1 \enspace.
\end{equation}
Then $V$ is a geometric drift in the sense of \eqref{eq:driftgerd} for the \plusES-ES.
\end{lemma}
\mathnote{in Geometric Convergence and Central Limit Theorems for Multidimensional Hastings and Metropolis Algorithms Author(s): G. O. Roberts and R. L. Tweedie, a if and only if condition for geometric ergodicity is stated with a lim sup $PV(\z)/V(\z)$.}
\begin{proof}
According to Lemma~\ref{lem:smallset}, the sets $\DD=\{ \z \in \ZZ, l_{1} \leq f(\z) \leq l_{2} \}$ with $0 < l_{1} < l_{2}$ are small sets for $\Z$.
The limit \eqref{eq:toub} (left) gives that for all $\epsilon > 0$ small enough, there exists $a_{2}$ such that for $\| \z\| \geq a_{2}$, $PV(\z) \leq (1 - \epsilon) V(\z) $. According to \eqref{eq:LB-UB} it implies the existence of $l_{2}$ such that for all $\z$ with $f(\z) \geq l_{2}$, $PV(\z) \leq (1 - \epsilon) V(\z) $. Similarly, the limit \eqref{eq:toub} (right)  gives  that for all $\epsilon > 0$ small enough, there exists $a_{1} > 0$ such that for $\| \z \| \leq a_{1}$, $PV(\z) \leq (1 - \epsilon) V(\z)$.  According to \eqref{eq:LB-UB} it implies the existence of $l_{1}$ such that for all $\z$ with  $f(\z) \leq l_{1}$, $PV(\z) \leq (1 - \epsilon) V(\z)$. Hence taking $\vartheta = 1-\epsilon$ for epsilon small enough, we have that outside the small set $\DD$, $PV(\z) \leq \vartheta V(\z)$. 
\end{proof}

\paragraph{Technical Results}
Before to establish the main proposition of this section, we derive a few technical results.
\begin{lemma}\label{lem:yop}
Assume that $f$ is continuous on $\R^{n}$ and for all $\n \in \R^{n}$, $f(\n) > f(0)$ then
$
\lim_{\| \z \| \to 0} 
\Pr(f(\z+\aN) > f(\z)) = 1 
\enspace,
$
where $\aN \sim \aN(0,\Id)$.
\end{lemma}
\begin{proof}
We express the probability $ \Pr(f(\z+\aN) > f(\z)) $ using the density of $\aN$:
$$
\Pr(f(\z+\aN) > f(\z))  = \int 1_{\{f(\z+\n) - f(\z) > 0\}}(\n) \pp(\n) d\n \enspace.
$$
For all $\n$ except for $\n=0$, $ 1_{\{f(\z+\n) - f(\z) > 0\}}(\n) $ converges to $1_{\{f( \n) - f(0) > 0 \}}(\n)$ when $\z$ goes to $0$ (the function $t \mapsto 1_{\{ t > 0 \}}(t) $ being discontinuous in $t=0$, for $\n=0$, for $\z$ to $0$, we arrive at the discontinuity point of the indicator, hence we cannot conclude about the limit). Hence by the dominated convergence theorem, we find that $\Pr(f(\z+\aN) > f(\z))$ converges to $\Pr(f(\aN) > f(0)) = 1$.
\end{proof}
\begin{lemma}\label{lem:limitps}
Assume that $f$ is a positively homogeneous function with degree $\alpha$ and $f(\x) >0$ for all $\x \neq 0$  and assume that $f$ is continuously differentiable. Then
\begin{align}\label{eq:loop1}
& \lim_{\|\z\| \to \infty} \Pr ( f(\z + \aN) > f(\z) )  = \frac{1}{2} \\\label{eq:loop2}
&\lim_{\|\z\| \to \infty} \Pr \left( \{ f(\z + \aN) \leq f(\z) \} \cap \{ f(\z+\aN) \geq \factonefifth^{\alpha} \} \right)  = \frac{1}{2}
\end{align}
where $\aN \sim \aN(0,\Id)$.
\end{lemma}
\begin{proof}
\newcommand{\g}{g_{\ref{lem:limitps}}}
\newcommand{\hh}{h_{\ref{lem:limitps}}}
We investigate first the limit \eqref{eq:loop1} and want to prove that 
\begin{equation}\label{eq:topp}
\forall \epsilon, \exists\ T > 0, \mbox{ such that for all } \z \mbox{ with } \| \z \| > T, | \Pr( f(\z + \aN) > f(\z) ) - \frac12 | < \epsilon \enspace.
\end{equation}
Let us fix one arbitrary $\epsilon$ for the rest of the proof and use the homogeneity property to write
$ \Pr ( f(\z + \aN) > f(\z) ) = \Pr( f(\frac{\z}{f(\z)^{1/\alpha}} + \frac{\aN}{f(\z)^{1/\alpha}} ) > 1 )$. Since $f$ is continuously differentiable, the mean value theorem gives us the existence for all $\n \in \R^{\dim}$ of $c_{\n} \in [0,1]$ such that
\begin{equation}\label{zoup1}
f\left( \frac{\z}{f(\z)^{\frac1\alpha}} + \frac{\n}{f(\z)^{\frac1\alpha}} \right) = \underbrace{f\left(\frac{\z}{f(\z)^{\frac1\alpha}}\right)}_{=1} + \nabla f \left( \frac{\z}{f(\z)^{\frac1\alpha}} + c_{\n} \frac{\n}{f(\z)^{\frac1\alpha}} \right) . \frac{\n}{f(\z)^{\frac1\alpha}} \enspace.
\end{equation}
The event $\{ f(\frac{\z}{f(\z)^{\frac{1}{\alpha}}} + \frac{\aN}{f(\z)^{\frac{1}{\alpha}}} ) > 1 \}$ thus equals  $\{ \nabla f \left( \frac{\z}{f(\z)^{\frac1\alpha}} + c_{\aN} \frac{\aN}{f(\z)^{\frac1\alpha}} \right) . \frac{\aN}{f(\z)^{\frac1\alpha}} > 0 \}$ also equal to the event $\{ \nabla f \left( \frac{\z}{f(\z)^{\frac1\alpha}} + c_{\aN} \frac{\aN}{f(\z)^{\frac1\alpha}} \right) . \aN > 0 \} $. Let us define the function $\g: \LL_{1} \times [0,1] $ as follows 
\begin{equation}\label{eq:fun1}
\g(\uu,v) = \Pr \left( \nabla f (\uu + c_{\aN} v \aN). \aN > 0 \right)
\end{equation}
such that $\Pr(f(\z + \aN) > f(\z)) = \g \left( \frac{\z}{f(\z)^{1/\alpha}}, \frac{1}{f(\z)^{1/\alpha}} \right) $. (Given the definition domain of $\g$ we have assumed that $\z$ is large enough such that $f(\z) \geq 1$.) We now prove the continuity of $\g$ that we express with its integral form as
$$
\g(\uu,v) = \int 1_{\{ \nabla f (\uu + c_{\n} v \n). \n > 0 \}}(\n) \pp(\n) d \n \enspace.
$$
Because we have assumed that the differential of $f$ is continuous, for all $\n$, the function $(\uu,v) \mapsto \nabla f (\uu + c_{\n} v \n). \n $ is continuous. The indicator function has one discontinuity point that could be reached if $\nabla f (\uu + c_{\n} v \n). \n = 0$. We thus exclude the point $\n=0$. In addition, with Property~\eqref{prop-grad},  $\nabla f(\uu + c_{\n} v \n) \neq 0$ if $\uu + c_{\n} v \n \neq 0$. Hence, given $(\uu_{0},v_{0})$ where we want to prove the continuity of $\g$, let  $\mathfrak{N}=\{ \n | \n = \alpha \uu_{0}, \alpha \in \R \}$ (this is a set of null measure provided $n \geq 2$) then for all $\n \in \R^{n} \backslash \mathfrak{N}$, the function $(\uu,v) \mapsto 1_{\{ \nabla f (\uu + c_{\n} v \n). \n > 0 \}}(\n) \pp(\n)$ is continuous in $(\uu_{0},v_{0}) $ and by the dominated convergence theorem we deduce the continuity of $\g(\uu,v)$ on  $\LL_{1} \times [0,1]$. By symmetry of $\pp(\n)$, for all $\uu$, 
$$
\g(\uu,0)= \int 1_{\{ \nabla f(\uu).\n > 0 \}}(\n) \pp(\n) d\n = 1/2 \enspace.
$$
Since $\g$ is continuous on a compact, it is uniformly continuous and hence there exists $\beta > 0$ such that for all $v \leq \beta$, and $\uu \in \LL_{1}$,  $|\g(\uu,v) - \g(\uu,0)| \leq \epsilon$ \enspace.
Taking $T' = \frac{1}{\beta}$, we then have that if $f(\z)^{1/\alpha} \geq T' $, $|\g(\frac{\z}{f(\z)^{1/\alpha}},\frac{1}{f(\z)^{1/\alpha}}) - \frac{1}{2} | \leq \epsilon$. From Lemma~\ref{lem:too}, we find that if $\| \z \| \geq T:=T'/M^{1/\alpha} $, then $f(\z)^{1/\alpha} \geq T' $ and $|\g(\frac{\z}{f(\z)^{1/\alpha}},\frac{1}{f(\z)^{1/\alpha}}) - \frac{1}{2} | \leq \epsilon$. Hence we have proven \eqref{eq:topp} that proves \eqref{eq:loop1} in the case $n \geq 2$. The case $n=1$ is even simpler and boils down to look at the limit when $v$ goes to $0$ of $\int 1_{\{ f( \uu + v \n ) >1  \}}(\n) p(\n) d \n = \int_{\R^{+}} 1_{\{ f'(\uu) + o(1) > 0  \}}(\n) p(\n) d \n  + \int_{\R^{-}} 1_{\{ - f'(\uu) + o(1) > 0  \}}(\n) p(\n) d \n$. Using the dominated convergence theorem, this latter limit equals $1/2$.\\
\noindent In a similar manner we investigate the limit \eqref{eq:loop2}. Using \eqref{zoup1}, we define in a similar manner on $\LL_{1} \times [0,1( $ the function
$$
\hh(\uu,v) = \int 1_{\{ \nabla f ( \uu + c_{\n} v \n). \n < 0 \}}(\n) 1_{\{ f( \uu/v + \n) \geq \factonefifth^{\alpha} \}}(\n) \pp(\n) d \n \enspace.
$$
Similarly we prove the continuity of $\hh$ on $\LL_{1} \times )0,1( $ and prolong it by continuity for $v = 0$ using the fact that for all $\n$, $\lim_{v \to 0} 1_{\{ f( \uu/v + \n) \geq \factonefifth^{\alpha} \}} = 1 $. We find then that
$$
\hh(\uu,0) = \int 1_{\{ \nabla f ( \uu). \n < 0 \}}(\n) \pp(\n) d \n = \frac12 \enspace.
$$
Like in the previous case we prove that for all $\epsilon$, there exists $T>0$ such that
$$
\mbox{for all } \| \z \| > T, |  \Pr \left( \{ f(\z + \aN) \leq f(\z) \} \cap \{ f(\z+\aN) \geq \factonefifth^{\alpha} \} \right)  - \frac12 | < \epsilon
$$ 
that proves \eqref{eq:loop2}. We omit the details as the proof follows the same lines as before.
\end{proof}

\begin{lemma}\label{lem-integra-V}
Let $f$ be a positively homogeneous function with degree $\alpha$ satisfying $f(\x) > 0$ for all $\x \neq 0$. Assume that $f$ is continuous on $\Rnstar$. Let $\aN$ denote a standard multivariate normal distribution. Then for all $\z$
\begin{align}\label{eq:u2-1}
& \mbox{If } \alpha \leq 1, E \left[ f(\z+\aN) \right] \leq M (\| \z \| + E(\| \aN \|))^{\alpha}  < + \infty \enspace, \\\label{eq:u2-1prime}
& \mbox{If } \alpha \geq 1, E \left[ f(\z+\aN) \right] \leq M (\| \z \| + E[\| \N \|^{\alpha}]^{1/\alpha})^{\alpha} < + \infty \enspace.
\end{align}
If in addition, $\alpha \leq n$, then there exists a constant $c_{\ref{lem-integra-V}}$ such that for all $\z $
\begin{equation}\label{eq:u2-2}
E \left[  \frac{1}{f(\z+ \aN)} \right] < c_{\ref{lem-integra-V}} \enspace.
\end{equation}
Consequently if $\alpha \leq n$, the function $V(\z)= f(\z) 1_{\{ f(\z) \geq  1 \}} + \frac{1}{f(\z)} 1_{\{ f(\z) <  1 \}}$ satisfies  for all $\z \in \Rnstar$
\begin{equation}\label{eq:u2-3}
\int V(\y) P(\z, d \y) < \infty \enspace.
\end{equation}
\end{lemma}
\begin{proof}
We start by proving \eqref{eq:u2-1} and \eqref{eq:u2-1prime}. Note first that $E[\| \N \|^{\alpha}] < \infty$ for all $\alpha > 0$. \mathnote{is a consequence from the form the density of a multivariate normal distribution.} According to Lemma~\ref{lem:too}, $f(\z+ \aN) \leq M \| \z + \aN \|^{\alpha}$. From the triangle inequality we obtain
\begin{equation}\label{eq:TIne}
\| \z + \aN \|^{\alpha} \leq ( \| \z \| + \| \aN \|)^{\alpha} \enspace.
\end{equation}
For $\alpha < 1$, $x \in [0,+\infty] \mapsto x^{\alpha}$ being concave, we obtain from Jensen inequality that $E[ ( \| \z \| + \| \aN \|)^{\alpha}] \leq (E[\| \z \| + E(\| \aN\|)])^{\alpha} = (\| \z \| + E(\| \aN \|))^{\alpha} $ which achieves to prove \eqref{eq:u2-1} (the case for $\alpha=1$ being the equality case for the last equations).
For $\alpha \geq 1$, we can apply the Minkowski inequality stating that $E[( \| \z \| + \| \aN \|)^{\alpha}]^{1/\alpha} \leq E[\| \z \|^{\alpha}]^{1/\alpha} + E[\| \N \|^{\alpha}]^{1/\alpha} = \| \z \| + E[\| \N \|^{\alpha}]^{1/\alpha}$.  Hence $E( ( \| \z \| + \| \aN \|)^{\alpha}) \leq 
(\| \z \| + E[\| \N \|^{\alpha}]^{1/\alpha})^{\alpha} $. Overall using the upper bound on $f(\z+ \aN)$ and \eqref{eq:TIne}, we find \eqref{eq:u2-1prime}.
\mathnote{If $0 < r \leq s$,  $E[| X |^{s}] < \infty$, $\Rightarrow$ $E[|X|^{r}] < \infty$. This actually comes from the Hoelder inequality and using the fact that we have a finite measure. + Minkowski inequality. If $r \geq 1$, $E(| X + Y|^{r})^{1/r} \leq E[|X|^{r}]^{1/r} + E[|Y|^{r}]^{1/r} $ }
\mathnote{Proposition 4.5 in John K. Hunter, Measure Theory. Consider $f$ $g$ taking values in $[0 \infty]$, if $0 \leq f \leq g$ then $0 \leq \int f \leq \int g$.}
We prove now \eqref{eq:u2-2}. We are writing in the sequel integrals of positive functions that are possibly infinite. We will prove actually that the functions are integrable (and the integrals finite) and prove that we have a bound for the integral independent of $\z$. Using Lemma~\ref{lem:too}
\begin{equation}\label{youm1}
A_{\ref{lem-integra-V}}=E \left[  \frac{1}{f(\z+ \aN)} \right] \leq E \left[ \frac{1}{m \| \z +  \aN   \|^{\alpha}} \right]
= \frac{1}{m} \int_{\R^{n}} \frac{1}{\| \z + \y\|^{\alpha}} \pp(\y) d\y
\end{equation}
Using a change of variables
\begin{equation}\label{youm2}
\int_{\R^{n}} \frac{1}{\| \z + \y\|^{\alpha}} \pp(\y) d\y = \int_{\R^{n}} \frac{1}{\| \yy \|^{\alpha}} \pp( \yy - \z) d \yy \enspace.
\end{equation}
The previous integral is possibly infinite because the function $\yy \mapsto \frac{1}{\| \yy \|^{\alpha}}$ has a singularity in zero. Let us study the integrability close to zero, i.e., investigate
$$
\int_{\Ball(0,1)} \frac{1}{\| \y \|^{\alpha}} \pp( \y - \z) d \y \leq K_{\ref{lem-integra-V}} \int_{\Ball(0,1)} \frac{1}{\| \y \|^{\alpha}} d \y
$$
where $K_{\ref{lem-integra-V}} $ is an upper bound on the density $\pp$ (hence independent of $\z$).
Using spherical coordinates for $n\geq 2$
\begin{multline*}
\int_{\Ball(0,1)} \frac{1}{\| \y \|^{\alpha}} d \y = \int_{0}^{1} \frac{r^{n-1}}{r^{\alpha}}  d r \prod_{i=1}^{n-2} \int_{0}^{\pi} \sin^{n-1-i}(\varphi_{i}) d \varphi_{i} \int_{0}^{2 \pi} d \varphi_{n-1} \\ \leq 2 (\pi)^{n-1} \int_{0}^{1} r^{n-1-\alpha} d r \enspace.
\end{multline*}
The latter integral is finite for $\alpha + 1 -n \leq 1$, i.e., $\alpha \leq n$. For $n=1$ we directly obtain $\int_{-1}^{1} \frac{1}{| \y |^{\alpha}} d \y < \infty $ if $\alpha \leq 1$.
To prove that $A_{\ref{lem-integra-V}}$ is bounded for all $\z$ by a constant independent of $\z$, we write
\begin{align*}
\int_{\R^{n}} \frac{1}{\| \y \|^{\alpha}} \pp(\y - \z) d \y & \leq K_{\ref{lem-integra-V}} \int_{\R^{n}} \frac{1}{\| \y \|^{\alpha}}  1_{\{ \| \y \| \leq 1  \}} d \y + \int_{\R^{n}} \frac{1}{\| \y \|^{\alpha}} 1_{\{ \| \y \| \geq 1  \}} \pp(\y - \z) d \y \\
& \leq K_{\ref{lem-integra-V}} \int_{B(0,1)} \frac{1}{\| \y \|^{\alpha}} d \y + \underbrace{\int_{\R^{n}}  \pp(\y - \z) d \y}_{=1}
\end{align*}
Hence $\int_{\R^{n}} \frac{1}{\| \y \|^{\alpha}} \pp(\y - \z) d \y$ is bounded by a constant independent of $\z$ if $\alpha \leq n$. Using this with \eqref{youm1} and \eqref{youm2} proves that $A_{\ref{lem-integra-V}}$ is bounded by a constant independent of $\z$.\\
Finally we prove \eqref{eq:u2-3}: Using the expression of the Markov chain given in \eqref{eq:transitionZ} denoting $\aN$ a standard normal distribution we find that $\int V(\y) P(\z,d\y) =$
\begin{multline*}
 E\left[ f \left( \frac{\z+\aN}{\factonefifth} \right) 1_{\{ f(\z+\aN) \leq f(\z) \}} 1_{\{f(\frac{\z+ \aN}{\factonefifth}) \geq 1  \}} \right]
+ E \left[ \frac{1_{\{f(\z + \aN) \leq f(\z) \}} 1_{\{ f((\z + \aN)/\gamma ) < 1  \}}}{f\left(\frac{\z+\aN}{\factonefifth}\right)}  \right] + \\
E \left[ f\left(\frac{\z}{\factonefifth^{-1/q}}\right) 1_{\{ f(\z+\aN) > f(\z) \}} 1_{\{f(\z/\factonefifth^{-\frac1q}) \geq 1  \}} \right]
+
E \left[ \frac{1_{\{f(\z + \aN) > f(\z) \}} 1_{\{ f(\z/\factonefifth^{-1/q}) < 1  \}}}{f(\z/\factonefifth^{-1/q})}   \right] \\
\leq \frac{1}{\factonefifth^{\alpha}} E \left[ f(\z+ \aN) \right] + \frac{1}{\gamma^{-\frac{\alpha}{q}}} f(\z) + \factonefifth^{\alpha} E \left[ \frac{1}{f(\z+ \aN)} \right] + \factonefifth^{-\frac{\alpha}{q}} \frac{1}{f(\z)} \enspace.
\end{multline*}
Using now \eqref{eq:u2-1}, \eqref{eq:u2-1prime} and \eqref{eq:u2-2} in the previous inequality we find \eqref{eq:u2-3}.
\end{proof}

\paragraph{Sufficient conditions for geometric ergodicity}
We are now ready to establish the main result of this section, namely a sufficient condition for geometric ergodicity. We need to make some further assumptions on the objective function that we gather as Assumption~\ref{ass:f2}.
\begin{assumption}\label{ass:f2}
The function $f: \R^{n} \to [0,+\infty[ $ satisfies Assumptions~\ref{ass:f}, i.e., is a positively homogeneous function with degree $\alpha$ and $f(\x) > 0$ for all $\x \neq 0$.\\
The function $f$ is continuously differentiable and $\alpha \leq n$.  
There exists $k \in \Nplus$, $c_{0}, \ldots, c_{k}$ in $\R$ such that for all $\zzz \in \LL_{1}$, $\y \in \R^{n}$, $c_{\zzz}, c_{\y} \in [0,1]$
\begin{equation}\label{eq:assOPOmajgradient}
\| \nabla f(\zzz+ c_{\zzz} c_{\y} \y)  \|^{2} \leq c_{0} + \sum_{i=1}^{k} c_{i} \| \y \|^{i} \enspace.
\end{equation}
\end{assumption}
In the next lemma, we verify that convex-quadratic functions satisfy the previous assumptions if $n \geq 2$.
\begin{lemma}\label{lem:CQ}
Let $f(\x) = \frac12 \x^{T} H \x$ with $H$ symmetric positive definite. It satisfies Assumptions~\ref{ass:f2} if $n \geq 2$. 
\end{lemma}
\begin{proof}
\newcommand{\tb}{| \! | \! |}
The function $f$ is positively homogeneous with degree $2$. Hence to satisfy the assumption $\alpha \leq n$ we need $n \geq 2$. Moreover, it is continuously differentiable and homogeneous with degree $2$ and satisfies $\nabla f(\x) = H \x $. Hence $\| \nabla H (\zzz+\gamma_{0} c_{\y} \y) \|^{2} \leq \tb H \tb \|  \zzz+ c_{\zzz} c_{\y} \y \|^{2} \leq K ( \| \zzz \| + \| \y \|  )^{2} \leq K (K_{1} + \| \y \|)^{2} = K K_{1} + 2 K K_{1} \| \y \| + K K_{1}^{2} \| \y \|^{2}  $ where $\tb . \tb$ is the induced matrix norm associated to the euclidian norm $\| . \|$ and $K$ is a bound for $\tb H \tb$ and $K_{1}$ a bound for the elements of $\LLL_{1}$. Hence \eqref{eq:assOPOmajgradient} is satisfied with $k=2$.
\end{proof}

We are now ready to state the main result of this section.
\begin{theorem}\label{prop:OPO}
Consider $(\Xt,\st)_{t \in \NNN}$, a (1+1)-ES with generalized one-fifth success rule as defined in \eqref{eq:sampling}, \eqref{eq:update-mean} and \eqref{eq:update-ss} optimizing $h = g \circ f$ where $g \in \Monotone$ and $f: \R^{n} \to [0,+\infty[$ satisfies Assumptions~\ref{ass:f2}. Let $\Z=(\Zt=\Xt/\st)_{t \in \NNN}$ be the Markov chain associated to the (1+1)-ES optimizing $h$ defined in \eqref{eq:transitionZ}. Then the function
\begin{equation}\label{eq:driftOPO}
V(\z)= f(\z) 1_{\{ f(\z) \geq  1 \}} + \frac{1}{f(\z)} 1_{\{ f(\z) < 1 \}}
\end{equation}
satisfies a drift condition for geometric ergodicity (in the sense of \eqref{eq:driftgerd}) for the Markov chain $\Z$ if  $\gamma > 1$ and
\begin{equation}\label{eq:inc-linear}
\frac12 \left( \frac{1}{\factonefifth^{\alpha}} + \factonefifth^{\alpha/q} \right) < 1 \enspace.
\end{equation}
\end{theorem}
The theorem calls for a few remarks. The LHS of \eqref{eq:inc-linear} corresponds to the expectation of the step-size change to the $\alpha$ on a linear function (i.e., $f(\x) = \vea .  \x + b $ for $\vea \in \R^{\dim}$ and $b \in \R$). Using the notation $\etastar$ introduced in \eqref{eq:sschangeETA} for the step-size change, the condition \eqref{eq:inc-linear} requires that on linear function 
\begin{equation}\label{eq:inc-linearETA}
E[ 1/({\etastar})_{\rm linear}^{\alpha}] < 1 \enspace,
\end{equation}
that translates a step-size increase on linear function. This condition is similar to the one found to prove geometric ergodicity for the $(1,\lambda)$ with self-adaptation \cite{TCSAnne04}. For an algorithm without elitist selection, the condition \eqref{eq:inc-linearETA} is the only one formulated on the step-size change to guarantee geometric ergodicity. It  ensures that the limit of $PV/V$ is smaller $1$ for $\z$ to infinity (see proof). For the $(1+1)$-ES another condition appears due to the limit of $PV/V$ in zero (see the details in the proof) that is reflected in the fact that $\factonefifth^{-\alpha/4} < 1$, i.e., the step-size should decrease in case of failure. This translates for the one-fifth success rule into $\gamma > 1$. Note that we also need this condition $\gamma > 1$ for the irreducibility, the small sets and the aperiodicity.
\begin{proof}
Using the definition of $V$ we can write $PV(\z) = E[V(\Ztt) | \Zt=\z]$ as
$$ 
PV(\z) =E\left[f(\Ztt)1_{\{ f(\Ztt) \geq  1\}} + \frac{1}{f(\Ztt)} 1_{\{ f(\Ztt) <  1 \}}  | \Zt=\z \right] \enspace.
$$
According to Lemma~\ref{lem:CSdrift}, we need to study the limits of $PV(\z)/V(\z)$ for $\z$ to $0$ and to $\infty$.
\paragraph*{Investigating the limit of $PV/V$ for $\z$ to infinity}
We first investigate the limit for $\| \z \| $ to infinity and consider $\z$ large enough, in particular we can assume that 
\begin{equation}\label{eq:case1}
f(\z) \geq \max{\{1,\factonefifth^{\alpha/q}\}}
\end{equation}
and hence $V(\z) = f(\z)$. Then 
\begin{equation}\label{eq:driftzlarge}
\frac{PV(\z)}{V(\z)} = 
\underbrace{E\left[ \frac{f(\Ztt)1_{\{ f(\Ztt) \geq 1 \}}  }{f(\z)}  | \Zt=\z \right]}_{A(\z)}
+
\underbrace{\frac{E\left[\frac{1}{f(\Ztt)} 1_{\{f(\Ztt) \leq 1 \}}  | \Zt=\z \right]}{f(\z)}}_{B(\z)}
\end{equation}

\noindent Throughout this proof we will denote $\aN$ the multivariate normal distribution used at iteration $t$ to sample a new candidate solution. Namely the update for $\Zt$ reads:
$$ 
\boxed{\Ztt = \frac{\Zt+ \aN}{\as} 1_{\{ f(\Zt + \aN) \leq f(\Zt) \}} + \frac{\Zt}{\as^{-1/q}} 1_{\{ f(\Zt + \aN) > f(\Zt) \}}} \enspace .
$$
Let us first investigate the term $A(\z)$ introduced in \eqref{eq:driftzlarge}. It is equal to
\begin{align*}
&   E \left[ \frac{f(\frac{\z + \aN}{\as})}{f(\z)} 1_{\{ f(\z + \aN) \leq f(\z) \}}  1_{\{ f(\frac{\z + \aN}{\as}) \geq 1 \}}  \right] + E \left[ \frac{f(\frac{\z}{\as^{-\frac1q}})}{f(\z)} 1_{\{ f(\z + \aN) > f(\z) \}} \underbrace{1_{\{ f(\frac{\z}{\as^{-1/q}}) \geq 1\}}}_{=1 \text{ (see} \eqref{eq:case1})} \right] \\
& = \frac{1}{\factonefifth^{\alpha}} \underbrace{E \left[ f\left(\frac{\z}{f(\z)^{\frac1\alpha}}+  \frac{\aN}{f(\z)^{\frac1\alpha}}\right) 1_{\{ f(\z + \aN) \leq f(\z) \}}  1_{\{ f(\frac{\z +  \aN}{\as}) \geq 1 \}}  \right]}_{A_{1}} + \factonefifth^{\alpha/q} \underbrace{E \left[ 1_{\{ f(\z + \aN) > f(\z) \}}   \right]}_{A_{2}} \enspace.
\end{align*}
Using Lemma~\ref{lem:limitps} we obtain that $A_{2}$ converges to $1/2$ when $\z$ goes to $\infty$. Let us now handle the term $A_{1}$. Using the mean value theorem we have the existence of $c_{\aN} \in [0,1]$ such that 
\begin{equation}\label{eq:withmvt}
f\left( \frac{\z}{f(\z)^{\frac1\alpha}} +  \frac{\aN}{f(\z)^{\frac1\alpha}} \right) = \underbrace{f\left(\frac{\z}{f(\z)^{\frac1\alpha}}\right)}_{=1} + \nabla f \left( \frac{\z}{f(\z)^{\frac1\alpha}} + c_{\aN} \frac{ \aN}{f(\z)^{\frac1\alpha}} \right) . \frac{ \aN}{f(\z)^{\frac1\alpha}} \enspace.
\end{equation}
Hence the term $A_{1}$ can be decomposed in two terms:
\begin{multline}
A_{1} = \underbrace{E \left[ 
1_{\{ f(\z + \aN) \leq f(\z) \}}  1_{\{ f(\frac{\z +  \aN}{\as}) \geq 1 \}}  \right]}_{A_{11}}
+\\
\frac{1}{f(\z)^{\frac1\alpha}} \underbrace{ E \left[\nabla f \left( \frac{\z}{f(\z)^{\frac1\alpha}} + c_{\aN} \frac{ \aN}{f(\z)^{\frac1\alpha}} \right) . \aN 1_{\{ f(\z + \aN) \leq f(\z) \}}  1_{\{ f(\frac{\z +  \aN}{\as}) \geq 1 \}}  \right]}_{A_{12}} \enspace.
\end{multline}
The term $A_{11}$ equals
$$
A_{11} = \int 1_{\{ f(\z + \n) \leq f(\z) \}}(\n)  1_{\{ f(\frac{\z +  \n}{\as}) \geq 1 \}}(\n) \pp(\n) d\n \enspace.
$$
According to Lemma~\ref{lem:limitps}, the term $A_{11}$ converges to $1/2$ when $\z$ goes to $\infty$. Note that  for all $\n$, the indicator $1_{\{ f(\frac{\z +  \n}{\as}) \geq 1 \}}(\n)$ converges to $1$ for $\z$ to $\infty$.
We now take care of the term $A_{12}$ and prove that $| A_{12}|$ is bounded which will imply that $A_{12} \frac{1}{f(\z)^{1/\alpha}}$ converges to zero when $\z$ goes to $\infty$.
\begin{multline*}
|A_{12}| \leq E \left[ | \nabla f \left( \frac{\z}{f(\z)^{\frac1\alpha}} + c_{\aN} \frac{ \aN}{f(\z)^{\frac1\alpha}} \right) . \aN | \right]   \leq E \left[ \|\nabla f \left( \frac{\z}{f(\z)^{\frac1\alpha}} + c_{\aN} \frac{ \aN}{f(\z)^{\frac1\alpha}}  \right) \| \| \aN \| \right] \\
 \leq E \left[ \|\nabla f \left( \frac{\z}{f(\z)^{\frac1\alpha}} + c_{\aN} \frac{ \aN}{f(\z)^{\frac1\alpha}}  \right) \|^{2} \right]^{\frac12} 
\underbrace{E \left[ 
\| \aN \|^{2} \right]^{\frac12}}_{=\sqrt{n}} \nonumber \enspace .
\end{multline*}
For the last two inequalities we have applied Cauchy-Schwarz inequalities. We denote $\zzz=\frac{\z}{f(\z)^{1/\alpha}}$ and $c_{\zzz} = \frac{1}{f(\z)^{1/\alpha}} $ and apply \eqref{eq:assOPOmajgradient}. Hence we find
$$
E \left[ \|\nabla f \left( \zzz + c_{\aN} c_{\zzz} \aN  \right) \|^{2} \right]
\leq
c_{0} + \sum_{i=1}^{k} c_{i} E[ \| \aN \|^{i} ] 
=: M_{\ref{prop:OPO}}
$$
and it follows that  $|A_{12}| \leq \sqrt{M_{\ref{prop:OPO}} n} $. Hence
$$
\lim_{\| \z \| \to \infty} A(\z) = \frac12 \left( \frac{1}{\as^{\alpha}} + \as^{\alpha/q} \right) \enspace.
$$
We investigate now the term $B(\z)$ defined in \eqref{eq:driftzlarge}.
$$
B(\z) = \frac{1}{f(\z)} \underbrace{E \left[ \frac{1_{\{f(\z+\aN)\leq f(\z) \}} 1_{\{ f(\frac{\z+ \aN}{\as})  \leq 1 \}}}{f(\frac{\z+ \aN}{\as})}  \right]}_{B_{11}}
+ \frac{1}{f(\z)} \underbrace{E \left[  \frac{1_{\{f(\z+\aN)> f(\z) \}} 1_{\{ f(\z/\as^{-\frac1q}) \leq 1 \}} }{f(\z/\as^{-\frac1q})} \right]}_{=0 \mbox{ as }   f(\z/\as^{-\frac1q}) > 1  }
$$
Let us now take care of the term $B_{11}$  which is upper bounded by
$$
B_{11} \leq E \left[ \frac{1}{f(\frac{\z+ \aN}{\as})}  \right] = \as^{\alpha} E \left[ \frac{1}{f(\z+ \aN)} \right] < \as^{\alpha} c_{\ref{lem-integra-V}} \enspace,
$$
where for the latter term we have used Lemma~\ref{lem-integra-V}. Overall we find that
$$
0 \leq B(\z) < \frac{1}{f(\z)} \as^{\alpha} c_{\ref{lem-integra-V}} \xrightarrow[\|\z\| \to \infty]{} 0
$$
where the latter limit comes from the fact that $\frac{1}{f(\z)}$ converges to zero when $\z$ goes to infinity.
Overall we have proven that
\begin{equation}\label{eq:cond-infinity}
\lim_{\|\z\| \to \infty} \frac{PV(\z)}{V(\z)} =  \frac12 \left( \frac{1}{\as^{\alpha}} + \as^{\alpha/q} \right) \enspace.
\end{equation}

\newcommand{\AAA}{A^{f<1}}
\newcommand{\BBB}{B^{f<1}}
\paragraph{Investigating the limit of $PV/V$ for $\z$ to zero}
We now investigate the limit for $\| \z \|$ to zero and consider thus $\z$ small enough, in particular we can assume
$$
f(\z) < \min \{ 1 , \factonefifth^{\alpha/q}  \}
$$
and hence $1/V(\z) = f(\z) $.
The quantity $PV(\z)/V(\z)$ writes
$$
\frac{PV(\z)}{V(\z)} = 
\underbrace{f(\z) E \left[ f(\Ztt) 1_{\{f(\Ztt) \geq 1 \}} | \Zt=\z \right]}_{C(\z)}
+
\underbrace{E \left[ \frac{f(\z)}{f(\Ztt)} 1_{\{ f(\Ztt) <  1 \}} | \Zt = \z \right]}_{D(\z)} \enspace.
$$
Let us investigate the term $C(\z)$:
\begin{multline}
C(\z) = f(\z) E \left[ f\left(\frac{\z+  \aN}{\as}\right) 1_{\{f(\z+  \aN) \leq f(\z) \}} 1_{\{ f((\z+  \aN)/\as) \geq  1 \}} \right] \\ + f(\z) E \left[f(\as^{\frac1q} \z) 1_{\{ f(\z+\aN) > f(\z) \}} 1_{\{ f(\as^{\frac1q} \z) \geq  1 \}} \right]
\end{multline}
and hence
$$
0 \leq C(\z) \leq f(\z) E \left[ f\left(\frac{\z+  \aN}{\as}\right) \right] + f(\z)^{2} \as^{\alpha/q} \enspace.
$$
According to \eqref{eq:u2-1} and \eqref{eq:u2-1prime}, for $\| \z \|$ small enough (hence staying in a bounded region), $E \left[ f\left(\frac{\z+  \aN}{\as}\right) \right]$ is a bounded  function of $\z$ 
and thus $C(\z)$ converges to zero when $\z$ goes to $0$:
\begin{equation}
\lim_{\z \to 0} C(\z) = 0 \enspace.
\end{equation}
Let us investigate the term $D(\z)$:
\begin{align*}
D(\z) & = f(\z) E \left[ \frac{ 1_{\{f(\z+ \aN) \leq f(\z) \}} 1_{\{ f((\z+ \aN)/\as) < 1  \}}}{f((\z+ \aN)/\as)} \right] + 
f(\z) E \left[ \frac{1_{\{f(\z+ \aN) > f(\z) \}}}{f(\z/\as^{-1/q})}  \underbrace{1_{\{ f(\frac{\z}{\as^{-\frac1q}}) <  1 \}}}_{=1} \right] \\
 & = f(\z) E \left[ \frac{1_{\{f(\z+ \aN) \leq f(\z) \}} 1_{\{ f((\z+ \aN)/\as) < 1   \}}}{f((\z+ \aN)/\as)}  \right] + 
\as^{-\alpha/q} E \left[ 1_{\{f(\z+ \aN) > f(\z) \}} \right] \\
& = \underbrace{f(\z) \as^{\alpha} E \left[ \frac{1_{\{f(\z+ \aN) \leq f(\z) \}} 1_{\{ f((\z+ \aN)/\as) <  1  \}}}{f(\z+ \aN)}  \right]}_{D_{1}(\z)}
+  \as^{-\alpha/q} \Pr( f(\z+ \aN) > f(\z)  ) \enspace
\end{align*}
The term $D_{1}(\z)$ is upper bounded by  $ f(\z)  \as^{\alpha} E \left[ \frac{1}{f(\z+ \aN)}  \right]$ and using Lemma~\ref{lem-integra-V} we find that for all $\z$, $ E \left[ \frac{1}{f(\z+ \aN)}  \right]$ is bounded by a constant. Hence since $f(\z)$ converges to $0$ when $\z$ goes to $0$ (Lemma~\ref{lem:too}), so does $D_{1}(\z)$. Overall since according to Lemma~\ref{lem:yop}, $\Pr( f(\z+ \aN) > f(\z)  )$ converges to $1$ when $\z$ goes to $0$, we find that
\begin{equation}
\lim_{\z \to 0} D(\z) = \as^{-\alpha/q} \enspace.
\end{equation}
Overall, we have proven that 
\begin{equation}\label{eq:cond-zero}
\lim_{\z \to 0} PV(\z)/V(\z) = \as^{-\alpha/q}  \enspace.
\end{equation}
According to Lemma~\ref{lem:CSdrift}, we obtain a drift condition for geometric ergodicity if the limits in \eqref{eq:cond-infinity} and \eqref{eq:cond-zero} are strictly smaller $1$, i.e., if  
$$
\frac12 \left( \frac{1}{\as^{\alpha}} + \as^{\alpha/q} \right) < 1
$$
and $\as^{-\alpha/q} < 1$. This latter condition being equivalent to $\gamma > 1$.
\end{proof}

\subsection{Harris Recurrence and Positivity}

Harris recurrence and positivity of the chain $\Z$ follow from the geometric drift proven in Theorem~\ref{prop:OPO}. Indeed, remind the following drift result for Harris recurrence:

\begin{theorem}[Theorem 9.1.8. in \cite{Tweedie:book1993}](Drift condition for Harris recurrence)
Suppose $\Z$ is a $\psi$-irreducible chain. If there exists a petite set $C$ and a function $V$ which is unbounded off petite sets such that 
\begin{equation}\label{eq:DriftHR}
\Delta V(\z) \leq 0, \z \in C^{c}
\end{equation}
holds, then $\Z$ is Harris recurrent.
\end{theorem}

In the previous theorem, a function $V: \ZZ \mapsto \R_{+}$ is unbounded off petite sets for $\Z$ if for any $c < \infty$, the sublevel sets $\LL_{c} = \{ \y : V(\y) \leq c \} $ is petite (see \cite[Section 8.4.2]{Tweedie:book1993}).
From \cite[Theorem~10.4.4]{Tweedie:book1993} a recurrent chain admits an unique (up to constant multiples) invariant measure. The positivity is deduced from another drift condition as expressed in the following theorem.
\begin{theorem}[From Theorem 13.0.1 in \cite{Tweedie:book1993}]
Suppose that $\Z$ is an aperiodic Harris recurrent chain with invariant measure $\pi$. The following are equivalent:\\
The chain is positive Harris: that is, the unique invariant measure $\pi$ is finite. \\
There exists some petite set $C$, some $b < \infty$ and a non-negative function $V$ finite at some $\z_{0} \in \ZZ$, satisfying
\begin{equation}\label{eq:driftPOS}
\Delta V(\z) \leq -1 + b 1_{C}(\z), \z \in \ZZ \enspace.
\end{equation}
\end{theorem}
Using those two theorems we deduce the corollary that under the conditions of Theorem~\ref{prop:OPO} the chain $\Z$ is positive Harris recurrent:
\begin{corollary}\label{coro:HR-P}
Consider a (1+1)-ES with generalized one-fifth success rule as defined in \eqref{eq:sampling}, \eqref{eq:update-mean} and \eqref{eq:update-ss} optimizing $h = g \circ f$ where $g \in \Monotone$, $f: \R^{n} \to [0,+\infty[$ satisfies Assumptions~\ref{ass:f2}.
Let $\Z=(\Zt=\Xt/\st)_{t \in \NNN}$ be the Markov chain associated to the (1+1)-ES optimizing $h$ defined in \eqref{eq:transitionZ}.
If $\gamma > 1$ and $\frac12 \left( \frac{1}{\factonefifth^{\alpha}} + \factonefifth^{\alpha/q} \right) < 1$, then $\Z$ is positive Harris recurrent.
\end{corollary}
\begin{proof}
The assumptions on $f$ ensure that the chain is $\varphi$-irreducible and aperiodic (see Section~\ref{sec:ISSA}) and the geometric drift function exhibited in Theorem~\ref{prop:OPO} satisfies obviously \eqref{eq:DriftHR} and \eqref{eq:driftPOS}.
It is unbounded off petite sets as the sublevel sets are small sets for $\Z$.
\end{proof}

\section{Linear Convergence of the $(1+1)$-ES with the Generalized One-fifth Success Rule}\label{sec:linear-convergence}

Using the properties derived on the normalized chain $(\Zt)_{t \in \NNN}$ we can now prove the global linear convergence of the $(1+1)$-ES with generalized one-fifth success rule. Linear convergence is formulated almost surely and in expectation. In a last part we characterize how fast the stationary regime where linear convergence takes place is reached. Common to the linear convergence results is the integrability of $\z \mapsto \ln \| \z \|$ with respect to the invariant probability measure $\pi$ of the chain $\Z$ that we investigate in the next section.

\subsection{Integrability w.r.t. the Stationary Measure}
To verify the integrability of $\z \mapsto \ln \| \z \|$ with respect to the invariant probability measure $\pi$, we use \eqref{eq:C-Vnorm} which is a consequence of the existence of a geometric drift. More formally we derive the following general technical  lemma.
\begin{lemma}\label{lem:intV}
Let $V$ be a geometric drift function for $\Z$ in the sense of \eqref{eq:driftgerd} and $\pi$ its invariant probability measure. Assume that there exists $\z \in S_{V}$ such that $\int  V(\y) P(\z,d\y) < \infty$, then $V$ is integrable against $\pi$:
$
\| \pi \|_{V} = \int V(\z) \pi(d \z) < + \infty \enspace.
$
\end{lemma}
\begin{proof}
From the inequality \eqref{eq:C-Vnorm}, there exists $R< \infty$ and $\rho <1$ such that for any $\z \in S_{V}$
\begin{equation}\label{totoo}
\| P(\z, .) - \pi \|_{V} \leq \rho R V(\z) \enspace.
\end{equation}
Consider a sequence of simple positive functions $V_{k}$ such that for each $\z$, $V_{k}(\z)$ converges to $V(\z)$ and $V_{k}(\z)$ is increasing. Then we know that $\int V(\z) \pi(d \z) = \lim_{k} \int V_{k}(\z) \pi (d \z)$ where the latter limit always exist but may be infinite.
From the triangular inequality we deduce that for all $k$
\begin{align}\label{eq:toppo}
 \int  V_{k}(\y) \pi(d\y) & =  \int  V_{k}(\y) \pi(d\y)  - \int  V_{k}(\y) P(\z,d\y) +  \int  V_{k}(\y) P(\z,d\y) \\
& \leq | \int  V_{k}(\y) \pi(d\y)  - \int  V_{k}(\y) P(\z,d\y)  | +  \int  V_{k}(\y) P(\z,d\y) \\
 & \leq  \| P(\z,.) - \pi  \|_{V}  + \int  V(\y) P(\z,d\y) 
\end{align}
where for the last inequality we have used the fact that $0 \leq V_{k} \leq V$ and the definition of $\| P(\z,.) - \pi  \|_{V}$ namely
$
\| P(\z,.) - \pi  \|_{V}  = \sup_{k , |k| \leq V} | \int k(\y) (P(\z,d\y) - \pi(d\y))  |.
$
The fact that $ \int  V_{k}(\y) P(\z,d\y) \leq \int  V(\y) P(\z,d\y) $ is a consequence of $ V_{k} \leq V$.
In addition, $\int  V(\y) P(\z,d\y) < \infty$ according to the assumptions. Using \eqref{totoo} we find that for all $k$
$$
\int  V_{k}(\y) \pi(d\y) \leq  R V(\z)  + \int  V(\y) P(\z,d\y) < \infty
$$
And hence $\int V \pi = \lim_{k} \int  V_{k}(\y) \pi(d\y) < \infty.$
\end{proof}

We can now apply the previous lemma to the drift function $V(\z) = f(\z) 1_{\{f(\z) \geq 1\}} + \frac{1}{f(\z)}1_{\{f(\z) < 1\}} $ assuming that  $f$ satisfies Assumptions~\ref{ass:f2} and that the sufficient conditions for a geometric drift of Theorem~\ref{prop:OPO} are satisfied (the fact that $PV(\z)$ is finite for one $\z$ in $S_{V}$ comes from Lemma~\ref{lem-integra-V}). We now prove that $| \ln \| \z \| |$ is upper bounded by a constant times $V(\z)$ that implies together with the previous lemma the integrability of $\ln \| \z \|$ w.r.t. stationary measure $\pi$.
\begin{lemma}\label{lem:integrability}
Assume that $f$ satisfies Assumptions~\ref{ass:f} and is continuous on $\Rnstar$. Then there exists a constant $\KK$ such that
\begin{equation}\label{homer}
| \ln \| \z \| | \leq \KK V(\z) \enspace.
\end{equation}
Assume that $f$ satisfies Assumptions~\ref{ass:f2} and that  that $\factonefifth > 1$ and $\frac12 (1/\factonefifth^{\alpha} + \factonefifth^{\alpha/q}) < 1 $. Then $\ln \| \z \|$ is integrable w.r.t. the stationary measure $\pi$ of $\Z$.
\end{lemma}
\begin{proof}
For $\| \z \|$ close to $0$, $f(\z)$ is close to $0$ and thus $V(\z) = \frac{1}{f(\z)}$. Hence there exists $c$ such that
$$
\frac{| \ln \| \z \| |}{V(\z)} = \frac{| \ln \| \z \| |}{f(\z)} \leq c \frac{| \ln \| \z \| |}{ \| \z \|^{\alpha}}
$$
where the latter term is bounded since $ \| \z \|^{\alpha} | \ln \| \z \| |$ goes to $0$ when $\z$ goes to $0$. Note that for the middle inequality we have used \eqref{eq:LB-UB}.
For $\z$ large, $V(\z) = f(\z) $ and using $| \ln \| \z \| | \leq \| \z \|^{\alpha} \leq f(\z)$ (again we use \eqref{eq:LB-UB}) we find that for $\z$ large $ | \ln \| \z \| | \leq V(\z)$.
Since $|\ln \| \z \||$ is continuous and hence bounded on all $\| \z\|$ in an interval $[a,b]$ with $0<a<b<+\infty$ then there exists $ \tilde{c}$ such that
$|\ln \| \z \|| 1_{\{ a \leq  \| \z \| \leq b \}} \leq \tilde{c} \leq \tilde{c} V(\z) 1_{\{ a \leq  \| \z \| \leq b \}} $. Overall we have proven that \eqref{homer} holds.

Since according to Lemma~\ref{lem-integra-V} $\int V(\y) P(\z,d\y) < + \infty$ for $\z \in \Rnstar$, when the conditions of Theorem~\ref{prop:OPO} are satisfied, we deduce from the previous lemma that $V$ is integrable w.r.t. $\pi$ and hence $| \ln \| \z \| |$ is integrable with respect to $\pi$.
\end{proof}

\subsection{Asymptotic Probability of Success}

We investigate now the asymptotic probability of success that comes into play in the convergence rate of the algorithm. Success is defined as whether a candidate solution is better than the current solution $\Xt$, i.e., as
$
P_{\frac{\x}{\sigma}} \left(f ( \X_{t} + \st \Ut^{1} ) \leq f( \X_{t}) \right)
$
and due to the scale-invariant property of $f$, this latter quantity can be expressed with the Markov chain $\Z$ as
$$
 P_{\z} \left( f( \Zt + \Ut^{1} ) \leq f(\Zt) \right) = E_{\z} \left[ 1_{\{ f( \Zt + \Ut^{1} ) \leq f(\Zt) \}} \right] \enspace.
 $$
The convergence of the probability of success is a consequence of the positivity and aperiodicity and can be deduced from \cite[Theorem~14.0.1]{Tweedie:book1993}.
\begin{proposition}[Asymptotic probability of success]
Let the $(1+1)$-ES with generalized one-fifth success rule optimize $h = g \circ f$ where $g \in \Monotone$ and $f$ satisfies Assumptions~\ref{ass:f2}. Assume that $\factonefifth > 1$ and $\frac12 (1/\factonefifth^{\alpha} + \factonefifth^{\alpha/q}) < 1 $. Let $\pi$ be the invariant probability measure of the normalized Markov chain $\Z$. 
Then for any initial condition $(\x,\sigma) \in \Rnstar \times \Rplusstar$, the following holds
\begin{equation}\label{eq:limProbaSuccess}
\PS:=
\lim_{t \to \infty} 
 P_{\frac{\x}{\sigma}} \left(f ( \X_{t} + \st \Ut^{1} ) \leq f( \X_{t}) \right)
= \int 1_{\{f(\y + \n) \leq f(\y)\}}(\y,\n) \pi(d\y) \pp(\n) d \n  \enspace.
\end{equation}
\end{proposition}
\begin{proof}
From \cite[Theorem~14.0.1]{Tweedie:book1993}, given a $\psi$-irreducible and aperiodic chain, given a function $k \geq 1$, if the chain is positive recurrent with invariant probability measure $\pi$ and $\pi(k) = \int \pi(d\z) k(\z) < \infty$, then for any $\z \in S_{\tilde V}= \{ \z : \tilde V(\z) < \infty \} $ where the function $\tilde V$ is an extended-valued function satisfying
\begin{equation}\label{eq:driftGEO-ERGO}
\Delta \tilde V(\z) \leq - k(\z) + b 1_{C}(\z) \enspace
\end{equation}
for some petite set $C$ and $b \in \R$, $\|P^{t}(\z,.) - \pi  \|_{k} \to 0$ holds. We take here $k(\z)=1$ and hence the geometric drift proven in Theorem~\ref{prop:OPO} implies also \eqref{eq:driftGEO-ERGO}. From Corollary~\ref{coro:HR-P}, the chain is positive and hence the function $k(\y)=1$ is integrable w.r.t. $\pi$.
Remark that 
$$
P_{\z} \left( f( \Zt + \Ut^{1} ) \leq f(\Zt) \right) 
= \int w(\y) P^{t}(\z,\y) d\y
$$
where $w(\y) = \int 1_{\{f(\y + \n) \leq f(\y) \}}(\n) \pp(\n) d \n  $, then $w(\y) \leq 1$ and hence from \cite[Theorem 14.0.1]{Tweedie:book1993} we deduce that ($\z = \x / \sigma$)
$$
| P_{\z} \left( f( \Zt + \Ut^{1} ) \leq f(\Zt) \right)  - \int w(\y) \pi(d \y) | \leq \| P^{t}(\z,.) - \pi \|_{\y \mapsto 1} \xrightarrow[t \to \infty]{} 0 \enspace.
$$
\end{proof}
We also derive a Law of Large Numbers for the asymptotic probability of success.
\begin{proposition}\label{prop:asCVPS}
Let the $(1+1)$-ES with generalized one-fifth success rule optimize $h = g \circ f$ where $g \in \Monotone$ and $f$ satisfies Assumptions~\ref{ass:f2}. Assume that $\factonefifth > 1$ and $\frac12 (1/\factonefifth^{\alpha} + \factonefifth^{\alpha/q}) < 1 $. Let $(\Zt=\Xt/\st)_{t \in \NNN}$ be the normalized Markov chain associated to $(\Xt,\st)_{t \in \NNN}$. Then for all initial condition $(\X_{0},\sigma_{0})$
\begin{equation}\label{eq:LLN-AS}
\frac{1}{t} \sum_{k=0}^{t-1} 1_{\{ f( \X_{k} + \sigma_{k} \U_{k}^{1} ) \leq f( \X_{k})  \}} 
=
\frac{1}{t} \sum_{k=0}^{t-1} 1_{\{ f( \Z_{k} + \U_{k}^{1} ) \leq f( \Z_{k})  \}} 
\xrightarrow[t \to \infty]{} \PS 
\end{equation}
where the asymptotic probability of success $\PS$ is defined in \eqref{eq:limProbaSuccess}.
\end{proposition}
\begin{proof}
From Corollary~\ref{coro:HR-P}, the chain is positive and Harris recurrent. The function $\y \mapsto w(\y) = \int 1_{\{f(\y + \n) \leq f(\y) \}}(\n) \pp(\n) d \n $ being integrable w.r.t.\ the stationary measure $\pi$, we can thus apply the Law of Large Numbers (\cite[Theorem 17.0.1]{Tweedie:book1993}) that gives us \eqref{eq:LLN-AS}.
\end{proof}

\subsection{Almost Sure Linear Convergence}

Almost sure linear convergence derives from the application of a Law of Large Number (LLN) (see Theorem~5.2 in \citecompanion). Some assumptions to be able to apply a LLN to $(\Zt)_{t \in \NNN}$ are
positivity, Harris-recurrence and integrability of $\ln \| \z \|$. We are then now ready to prove the almost sure linear convergence of the $(1+1)$-ES with generalized one-fifth success rule.

\begin{theorem}
Let the $(1+1)$-ES with generalized one-fifth success rule optimize $h = g \circ f$ where $g \in \Monotone$ and $f$ satisfies Assumptions~\ref{ass:f2}. Assume that $\factonefifth > 1$ and $\frac12 (1/\factonefifth^{\alpha} + \factonefifth^{\alpha/q}) < 1 $. Then for all initial condition $(\X_{0}, \sigma_{0})$ almost sure linear convergence for the mean vector and for the step-size holds, i.e.,
\begin{align}\label{eq:asCVx}
& \frac{1}{t} \ln \frac{\| \X_{t} \|}{\| \X_{0} \|}  \xrightarrow[t \to \infty]{} \ln \gamma \left( \frac{q+1}{q} \PS - \frac1q  \right)
\\\label{eq:asCVsigma}
& \frac{1}{t} \ln \frac{\sigma_{t} }{ \sigma_{0} }  \xrightarrow[t \to \infty]{} \ln \gamma \left( \frac{q+1}{q} \PS - \frac1q  \right) \enspace .
\end{align}
where $\PS$ is the asymptotic probability of success defined in \eqref{eq:limProbaSuccess}.
\end{theorem}
\begin{proof}
The proof follows the same line as the proof of Theorem~5.2 in \citecompanion.
We start by re-writing the log-progress:
\begin{multline*}
\frac{1}{t} \ln \frac{\| \X_{t} \|}{\| \X_{0} \|} = \frac1t \sum_{k=0}^{t-1} \ln \frac{\| \X_{k+1} \|}{\|  \X_{k} \|} 
= \frac1t \sum_{k=0}^{t-1} \ln \frac{\sigma_{k+1} \| \Z_{k+1}\| }{\sigma_{k} \| \Z_{k}  \|}
= \frac1t \sum_{k=0}^{t-1} \ln \frac{\GG_{2}(1,\Y_{k}) \| \Z_{k+1}\| }{ \| \Z_{k}  \|}  \\
= \underbrace{\frac1t \sum_{k=0}^{t-1} \ln \| \Z_{k+1}  \|}_{A} - \underbrace{\frac1t \sum_{k=0}^{t-1} \ln \| \Z_{k}  \|}_{B} + \underbrace{\frac1t \sum_{k=0}^{t-1}  \ln \GG_{2}(1,\Y_{k})}_{C}
\end{multline*}
where we have used the scale-invariant property \eqref{eq:SIG2} and the fact that $\Y_{k} = \Perm_{(\Xt,\st)}(\U_{k}) * \Ut  = \Perm_{(\Z_{k},1)}(\U_{k})
$.
Since $\ln \| \z \|$ is integrable w.r.t.\ the stationary measure $\pi$ we can apply the LLN to the terms $A$ and $B$ and we find that they both converge towards $\int \ln \| \z \| \pi(d \z)$ such that $A$ minus $B$ converges to zero. Let us investigate the term $C$, since
$
\GG_{2}(1,\Yt) = (\factonefifth - \factonefifth^{-1/q}) 1_{\{ f( \Z_{t} + \U_{t}^{1}) \leq f(\Z_{t}) \}} + \factonefifth^{-1/q}
$
we find that
$
\ln \GG_{2}(1,\Yt) =  \ln \factonefifth ( ( 1 + 1/q )   1_{\{ f( \Z_{t} + \U_{t}^{1}) \leq f(\Z_{t}) \}} - \frac1q ).
$
%
Therefore
\begin{multline}\label{eq:tchoumb}
C = \frac1t \sum_{k=0}^{t-1}  \ln \GG_{2}(1,\Y_{k}) =   \ln \factonefifth \left( 1 + \frac1q \right)  \frac1t \sum_{k=0}^{t-1} 1_{\{ f( \Z_{k} + \U_{k}^{1}) \leq f(\Z_{k}) \}}  - \frac1q  \ln \factonefifth  \xrightarrow[]{} \\ 
\ln \factonefifth \left( 1 + \frac1q \right) \int 1_{\{ f(\z + \uu) \leq f(\z)  \}}(\z,\uu) \pi( d \z) \pp(\uu) d \uu  - \frac1q  \ln \factonefifth = \ln \gamma \left[ \left(1+\frac1q\right) \PS - \frac1q \right]
\end{multline}
where we have used Proposition~\ref{prop:asCVPS} for the latter limit.
The limit \eqref{eq:asCVsigma} follows from the fact that 
$$
\frac1t \ln \frac{\sigma_{t}}{\sigma_{0}} = \frac1t \sum_{k=0}^{t-1} \ln \frac{\sigma_{k+1}}{\sigma_{k}} =  \frac1t \sum_{k=0}^{t-1} \ln  \GG_{2}(1,\Y_{k})
$$
with as above  $\Y_{k} = \Perm_{(\Xt,\st)}(\U_{k}) * \Ut  = \Perm_{(\Z_{k},1)}(\U_{k})
$. Using~\eqref{eq:tchoumb} we obtain \eqref{eq:asCVsigma}.
\end{proof}

We define the convergence rate as minus the almost sure limit of the logarithm of $\| \Xt \|$ or of $\sigma_{t}$ that corresponds to minus expectation of the logarithm of the step-size change w.r.t. the stationary distribution, i.e.,
\begin{equation}\label{eq:CR}
\boxed{\CR := - E_{\pi} \left[ \ln \etastar \right] 
=
- \ln \gamma   \left( \frac{q+1}{q} \PS - \frac1q	  \right) } \enspace.
\end{equation}
Figure~\ref{fig:simul} presents some convergence graphs of the $(1+1)$-ES with generalized $1/5$ success rule. The slope of the linear decrease observed in log scale (after a small adaptation period on the left graph) corresponds to $- \CR$.

\paragraph{Sign of the convergence rate}
Convergence will take place if $\CR > 0$. We prove in the next proposition an alternative expression for the convergence rate that allows us to conclude that $\CR > 0$.
\begin{proposition}\label{prop:SignCR}
Let the $(1+1)$-ES with generalized one-fifth success rule optimize $h = g \circ f$ where $g \in \Monotone$ and $f$ satisfies Assumptions~\ref{ass:f2}. Assume that $\factonefifth > 1$ and $\frac12 (1/\factonefifth^{\alpha} + \factonefifth^{\alpha/q}) < 1 $. Let $\CR$ be the convergence rate of the algorithm given in \eqref{eq:CR}. Then 
\begin{align}\label{eq:CR2}
\CR & = - \frac{1}{\alpha} E_{\pi} \left(\ln \frac{f(\Zt+\Ut^{1}1_{\{f(\Zt+\Ut^{1}) \leq f(\Zt)  \} }) }{f(\Zt)}  \right) \\
& =  - \frac{1}{\alpha} 
\int \ln 
\frac{f(\z+ \n 1_{f(\z+\n) \leq f(\z)})}
{f(\z)} 
\pp(\n) p_{\pi}(\z) d\z d \n
\end{align}
where $p_{\pi}$ is the density of the invariant probability measure $\pi$ with respect to the Lebesgue measure.
Consequently $\CR > 0$, i.e., linear \emph{convergence} indeed takes place.
\end{proposition}
\begin{proof}
Because $\pi$ is the invariant probability measure of $\Z$, if $\Z_{0} \sim \pi$ then for all $t$, $\Zt \sim \pi$ such that
$
E_{\pi} \left[ \ln \frac{f(\Ztt)}{f(\Zt)} \right] 
= 0.
$
On the other hand, since $\Ztt = (\Zt+\Ut^{1}1_{\{f(\Zt+\Ut^{1}) \leq f(\Zt)  \} })/\etastar$ we deduce that
$$
E_{\pi} \left[ \ln \frac{f(\Ztt)}{f(\Zt)} \right]  = E_{\pi} \left[ - \alpha \ln \etastar + \ln \frac{f(\Zt+\Ut^{1}1_{\{f(\Zt+\Ut^{1}) \leq f(\Zt)  \} }) }{f(\Zt)} \right] = 0
$$
Thus
\begin{align}
\CR = - E_{\pi} \left[ \ln \etastar \right] & = - \frac{1}{\alpha} E_{\pi} \left(\ln \frac{f(\Zt+\Ut^{1}1_{\{f(\Zt+\Ut^{1}) \leq f(\Zt)  \} }) }{f(\Zt)}  \right) \\
& =  - \frac{1}{\alpha} 
\int \ln 
\frac{f(\z+ \n 1_{f(\z+\n) \leq f(\z)})}
{f(\z)} 
\pp(\n) p_{\pi}(\z) d\z d \n
\end{align}
where in the previous equation we have used the fact that according  to \cite[Theorem~10.4.9]{Tweedie:book1993}, $\pi$ and the maximal irreducibility measure for $\Z$ are equivalent. Hence $\pi$ is equivalent to the Lebesgue measure. We denoted $p_{\pi}$ its density.
Since $\frac{f(\z+ \n 1_{f(\z+\n) \leq f(\z)})}
{f(\z)} \leq 1 $ we see that $\CR \geq  0$. However $\CR =0$ is impossible as it would imply that $\ln 
\frac{f(\z+ \n 1_{f(\z+\n) \leq f(\z)})}
{f(\z)} 
\pp(\n) p_{\pi}(\z) = 0$ almost everywhere.
\end{proof}

The fact that the convergence rate $\CR$ is strictly positive is equivalent to having the asymptotic probability of success satisfying $\PS < 1/(q+1)$. In the case where the target probability of success is $1/5$ (as proposed in \cite{Schumer:68,Rechenberg}), this implies $\PS < 1/5$. Hence we find that when convergence occurs the asymptotic probability of success is strictly smaller than $1/5$.

In the case of a non elitist algorithm, it is not easy to obtain the sign of the convergence rate and one needs to resort to numerical simulation (see \cite{TCSAnne04}).

\subsection{Linear convergence in expectation}

Linear convergence in expectation is formulated in the next theorem. The proof follows the lines of Theorem~5.3 in \citecompanion.
\begin{theorem}
Let the $(1+1)$-ES with generalized one-fifth success rule optimize $h = g \circ f$ where $g \in \Monotone$ and $f$ satisfies Assumptions~\ref{ass:f2}. Assume that $\factonefifth > 1$ and $\frac12 (1/\factonefifth^{\alpha} + \factonefifth^{\alpha/q}) < 1 $. Then for all initial condition $(\X_{0}, \sigma_{0})=(\x_{0},\sigma_{0}) \in \Rnstar \times \Rplusstar$
\begin{equation}\label{eq:convEXPX}
\lim_{t \to \infty} E_{\frac{\x_{0}}{\sigma_{0}}} \left[\ln \frac{\|\Xtt\|}{\| \Xt \|} \right] = - \CR
\end{equation}
and
\begin{equation}\label{eq:convEXPsigma}
\lim_{t \to \infty} E_{\frac{\x_{0}}{\sigma_{0}}} \left[\ln \frac{\sigma_{t+1}}{\sigma_{t}}\right] = - \CR \enspace.
\end{equation}
\end{theorem}
\begin{proof}
The conditions of Theorem~5.3 of \citecompanion\ are satisfied. Hence we can conclude to the linear convergence in expectation for all initial condition $(\x_{0},\sigma_{0})$ such that $V(\x_{0}/\sigma_{0}) < \infty$ where $V$ is a function satisfying 
\begin{equation}\label{eq:toptop}
\Delta V(\z) \leq - ( | \ln \| \z \| | + 1) + b 1_{C}(\z) \enspace.
\end{equation}
Let us show that the previous condition is satisfied for a function proportional to the geometric drift function of Theorem~\ref{prop:OPO}. This will hence imply that the initial condition $(\x_{0},\sigma_{0})$ can be taken in $\Rnstar \times \Rplusstar$.

Indeed, consider the function $\tilde V$ given in \eqref{eq:driftOPO} (to avoid ambiguity we denote $\tilde V$ the function originally denoted $V$). We have proven that it satisfies a drift condition for geometric ergodicity, that is, there exists some petite set $C$, some constants $b< \infty$ and $0< \vartheta <1$ such that
\begin{equation}\label{eq:intern1}
\Delta \tilde V \leq (\vartheta -1) \tilde V(\z) + b 1_{C}(\z) \enspace.
\end{equation}
Since in Lemma~\ref{lem:integrability} we have proven that $|\ln\| \z \| | \leq \KK \tilde V(\z)$ and $\tilde V \geq 1$, the following inequality holds: $ |\ln\| \z \| | + 1 \leq (\KK + 1) \tilde V(\z)$. We deduce that $(\vartheta - 1) (\KK+1) \tilde V(\z) \leq (\vartheta - 1) (| \ln \| \z \| | + 1) $ and hence 
\begin{equation}\label{eq:intern2}
(\vartheta -1) (\KK+1)\tilde V(\z) /(1 - \vartheta)  \leq - (|\ln \| \z \| |+1)
\end{equation}
Let us take $V = (\KK+1)/(1 - \vartheta) \tilde V$, \eqref{eq:intern1} implies that
$$
\Delta V \leq (\vartheta -1) V(\z) + b \frac{\KK+1}{1 - \vartheta} 1_{C}(\z) \leq - (|\ln \| \z \||+1) + b \frac{\KK+1}{1 - \vartheta} 1_{C}(\z)
$$
where we have used \eqref{eq:intern2} for the latter inequality. Since $V(\z) < \infty$ whenever $f(\z) < \infty$ and $1/f(\z) < \infty$ we deduce from Theorem~5.3 in \citecompanion\ that the limits \eqref{eq:convEXPX} and \eqref{eq:convEXPsigma} hold for all $\x_{0}/\sigma_{0}$ such that  $f(\x_{0}/\sigma_{0}) < \infty$ and $1/f(\x_{0}/\sigma_{0}) < \infty$, i.e., for all $(\x_{0},\sigma_{0}) \in \Rnstar \times \Rplusstar$.
\end{proof}

\subsection{Consequences of Geometric Ergodicity: Adaptivity at a Geometric Rate}

The geometric ergodicity translates that the invariant probability distribution is reached geometrically fast. It implies that from any starting point in $\Rnstar \times \Rplusstar$, the expected (step-size) log progress $E_{\frac{\x_{0}}{\sigma_{0}}} [ \ln \frac{\stt}{\st} ] $ approaches the convergence rate $\CR$ geometrically fast. More precisely we have the following theorem:
\begin{theorem}
Let the $(1+1)$-ES with generalized one-fifth success rule optimize $h = g \circ f$ where $g \in \Monotone$ and $f$ satisfies Assumptions~\ref{ass:f2}. Assume that $\factonefifth > 1$ and $\frac12 (1/\factonefifth^{\alpha} + \factonefifth^{\alpha/q}) < 1 $. Then there exists $r > 1$ and $R < \infty$ such that for all initial condition $(\X_{0}, \sigma_{0})=(\x_{0},\sigma_{0}) \in \Rnstar \times \Rplusstar$
\begin{equation}
\sum_{t} r^{t} |E_{\frac{\x_{0}}{\sigma_{0}}} \ln \frac{ \stt }{\st} - \CR | \leq R V(\x_{0}/\sigma_{0}) \enspace.
\end{equation}
This equation implies in particular that for any initial condition $(\x_{0},\sigma_{0}) \in \Rnstar \times \Rplusstar $
\begin{equation}
|E_{\frac{\x_{0}}{\sigma_{0}}} \ln \frac{ \stt }{ \st } - \CR  | r^{t} \xrightarrow[t \to \infty]{} 0
\end{equation}
where $r$ is independent of the starting point.
\end{theorem}
\begin{proof}
Under the assumptions of the theorem we have proven in Theorem~\ref{prop:OPO} that $V$ defined in \eqref{eq:driftOPO} satisfies a geometric drift function. Hence according to \eqref{eq:C-Vnorm} there exists $R>0$ and $r > 1$ such that for any starting point $\z_{0}$ in the set $S_{V}= \{ \z : V(\z) < \infty \} $
\begin{equation}\label{young}
\sum_{t} r^{t} \| P^{t}(\z_{0},.) - \pi \|_{V} \leq R V(\z_{0})
\end{equation}
where $\| \nu \|_{V} = \sup_{g: |g| \leq V} | \nu(g)|$. Remark that if $V$ satisfies a geometric drift condition, then $k V$ satisfies also a geometric drift condition for any constant $k \geq 1$.
Consider the function  $g(\z) = \ln \gamma \left( \frac{q+1}{q} E[1_{\{f(\z + \aN) \leq f(\z)  \}}] - \frac1q \right) $. Then $|g(\z)| \leq \ln \gamma (q+2)/q $ is bounded and hence $|g(\z)| \leq \underbrace{ \ln \gamma (q+2)/q}_{k}  V(\z)$. In addition, $E_{\z_{0}}[\ln \stt/\st]=E_{\z_{0}}[g(\Zt)] = P^{t}(\z_{0},.)(g)$ which yields 
\begin{multline*}
| E_{\z_{0}}\left[\ln \frac{\stt}{\st}\right] - \CR | =  | \int g(\z) P^{t}(\z_{0},d\z) - \int \pi(d\z) g(\z) | = \\ | (P^{t}(\z_{0},.) - \pi)(g)|  \leq  \| P^{t}(\z_{0},.) - \pi \|_{k V} \enspace
\end{multline*}
and thus according to \eqref{young} there exists $R> 0$ and $r >1$ such that for any $(\x_{0},\sigma_{0})$
\begin{equation}
\sum_{t} r^{t} | E_{\frac{\x_{0}}{\sigma_{0}}}\left[\ln \frac{\stt}{\st}\right] - \CR | \leq R V\left(\frac{\x_{0}}{\sigma_{0}} \right) \enspace.
\end{equation}
This equation implies in particular that, for any initial condition $(\x_{0},\sigma_{0})$
\begin{equation}
| E_{\frac{\x_{0}}{\sigma_{0}}}\left[\ln \frac{\stt}{\st}\right] - \CR | r^{t} \xrightarrow[t \to \infty]{} 0
\end{equation}
where $r$ is independent of the starting point. 
\end{proof}
\mathnote{It turns out to be a pain in the ass to prove that the function  $g(\z) = E[ \ln \| \z + \aN 1_{f(\z+\aN) \leq f(\z)} \| / \| \z \|] $ is dominated by $V$. Actually one way to prove it is to go via the continuity of $g$ and then use the limits for $\| \z \|$ to infinity and $\| \z \|$ to zero of $g/V$. However the continuity by itself is a pain in the ass because of the negative part of the logarithm that needs to be controlled. I am kind of convinced that the result is correct thought. But applying the Dominated convergence theorem is painful - There might be more powerful tools ?? Why not trying the classical theorem with continuity on compacts ? [ because $n$ does not live in a compact] use spherical coordinates??}

Remark that we only derived a result for the log-step-size progress in the previous theorem. We believe that a similar result for the log-progress $\ln \| \Xtt \|/\| \Xt \|$ can also be derived. However it appears to be more technical to control the corresponding $g$-function (see proof) by $V$.

Our last result derives from the Central Limit Theorem (CLT), which is also a consequence of the the geometric ergodicity.
 It describes the speed of convergence of the result obtained from applying the LLN. We remind first a CLT result for MC extracted from \cite{Tweedie:book1993}. Given a function $g$, $S_{t}(g) = \sum_{k=1}^{t} g(\Z_{k}) $.
\begin{theorem}[Theorem 17.0.1, Theorem 16.0.1 in \cite{Tweedie:book1993}]\label{theo:CLT}
Suppose that $\Z$ is a positive Harris chain with invariant probability measure $\pi$ and is aperiodic. Suppose that $\Z$ satisfies a geometric drift in the sense of \eqref{eq:driftgerd}. Let $g$ be a function on $\ZZ$ that satisfies $g^{2} \leq V$ and let $\bar g$ denote the centered function $\bar g = g - \int g d \pi$. Then the constant 
$$
\gamma_{g}^{2} = E_{\pi}[\bar g^{2}(\Z_{0}) ] + 2 \sum_{k=1}^{\infty} E_{\pi}[ \bar g(\Z_{0}) \bar g (\Z_{k})]
$$
\mathnote{Interesting to remark that the previous expression makes sense if the distribution of $(\Z_{0},\Z_{k})$ changes over time while we initialize $\Z_{0}$ with the invariant distribution $\pi$.}
is well defined, non-negative and finite, and coincides with the asymptotic variance
$$
\lim_{t \to \infty} \frac1t E_{\pi}[(S_{t}(\bar g))^{2}] = \gamma_{g}^{2} \enspace.
$$
If $\gamma_{g}^{2} > 0$ then the Central Limit Theorem holds for the function $g$, that is for any initial condition $\z_{0}$
$$
\lim_{t \to \infty} P_{\z_{0}} \left\{ (t \gamma_{g}^{2})^{-1/2} S_{t}(\bar g) \leq x  \right\} = \int_{- \infty}^{x} \frac{1}{\sqrt{2 \pi}} e^{-u^{2}/2} d u \enspace .
$$ 
If $\gamma_{g}^{2} = 0$, then $\lim_{t \to \infty} \frac{1}{\sqrt{t}} S_{t}(g) = 0 $ a.s.
\end{theorem}
\mathnote{There is a subtlety as we cannot take directly a deterministic $g$ function. We can consider the chain $(\Zt,\Ut^{1})$ instead, in which case the function $g(\z,\uu) = \frac{\z + \uu 1_{f(\z + \uu) \leq f(\z)}}{\etastar(\z,\uu)}$. Quel drift pour cette chaine ? $V(\z,\uu) = VV(\z)$ or $V(\z,\uu) = VV(\z) + \| \uu \| $ - On the other hand it might be for this reason that it was difficult to prove the result with the $g$ function in the previous theorem ?  }
\mathnote{This subtlety anyhow already appears when we apply the LLN, we hide it but still take the expectation w.r.t.. the independent random variable.}

\begin{theorem}\label{CLT:OPO}
Let the $(1+1)$-ES with generalized one-fifth success rule optimize $h = g \circ f$ where $g \in \Monotone$ and $f$ satisfies Assumptions~\ref{ass:f2}. Assume that $\factonefifth > 1$ and $\frac12 (1/\factonefifth^{\alpha} + \factonefifth^{\alpha/q}) < 1 $. Then for any initial condition $(\x_{0},\sigma_{0}) \in \Rnstar \times \Rplusstar$
$$
\lim_{t \to \infty} P_{\frac{\x_{0}}{\sigma_{0}}} \left\{ \frac{\sqrt{t}}{\gamma_{g}} \left( \frac{1}{t} \ln \frac{\st}{\sigma_{0}}  - \CR  \right)  \right\}  = \frac{1}{\sqrt{2\pi}} \int_{- \infty}^{x} \exp(- u^{2}/2) du
$$
where $\gamma_{g}^{2}= \frac1t E_{\pi}[(\ln \frac{\sigma_{t}}{\sigma_{0}} - t \CR)^{2}] > 0$.
\end{theorem}
\begin{proof}
We have seen that $\frac{1}{t} \ln \frac{\st}{\sigma_{0}} = \frac{1}{t} \sum_{k=0}^{t-1} \ln \frac{\sigma_{k+1}}{\sigma_{k}} = \frac{1}{t} \sum_{k=0}^{t-1} g_{\ref{CLT:OPO}}(\Z_{k},\U_{k}^{1})  $ with
$$
g_{\ref{CLT:OPO}}(\Z_{k},\U_{k}^{1}) = \ln \gamma \left( \left( \frac{1+q}{q} \right) 1_{ \{ f(\Z_{k}+\U_{k}^{1}) \leq f(\Z_{k}) \}} - \frac{1}{q} \right) \enspace.
$$
Because the definition domain of $g_{\ref{CLT:OPO}}$ is  $\Rnstar \times \R^{n}$, we cannot directly compare it to the geometric drift function $V$ and verify whether $g^{2} \leq V$. However let us consider not only $\Zt$ but the couple $(\Zt,\Ut^{1})$ where the $\Ut^{i}$ are i.i.d.\ distributed according to $\N(0,\Id)$. Clearly $(\Zt,\Ut^{1})$ is an homogeneous Markov Chain that will inherit the properties of $\Zt$. Typically the chain is $\varphi$-irreducible with respect to the Lebesgue measure on $\Rnstar \times \R^{n}$, aperiodic and $D_{[l_{1},l_{2}]} \times \R^{n}$ are some small sets for the chain. The function $\tilde V(\z,\uu) = V(\z)$ (with $V$ defined in \eqref{eq:driftOPO}) satisfies a geometric drift in the sense of \eqref{eq:driftgerd} as from the small set shape we see that we only need to control the chain outside $D_{[l_{1},l_{2}]}$ while $\uu \in \R^{n}$. The invariant probability distribution of the chain is $\pi  	\otimes p$.

To be able to apply Theorem~\ref{theo:CLT}, we need to verify that $g_{\ref{CLT:OPO}}^{2} \leq K_{\ref{CLT:OPO}} \tilde V$ with $K_{\ref{CLT:OPO}}$ a constant larger $1$ (as the same arguments used before holds here as well, namely if $\tilde V$ satisfies a geometric drift condition, then every multiple of $\tilde V$ also satisfies a geometric drift condition, provided the multiplication constant is larger one to still ensure that the function is larger $1$).
Let us now remark that $g_{{\ref{CLT:OPO}}}^{2} \leq ((\ln \gamma) (2+q)/q)^{2} $ and hence $g_{{\ref{CLT:OPO}}}^{2} \leq (((\ln \gamma) (2+q)/q)^{2} + 1) \tilde V$.
Using Theorem~\ref{theo:CLT}, we know that $\gamma_{g}^{2}$ is well defined and cannot equal $0$ otherwise it would imply that $\frac{1}{\sqrt{t}} \ln \frac{\st}{\sigma_{0}} = 0$ that would contradict the fact that $\lim_{t \to \infty}\frac{1}{t} \ln \frac{\st}{\sigma_{0}} = - \CR \neq 0$. Hence $\gamma_{g}^{2} > 0$ and we conclude using Theorem~\ref{theo:CLT}.
\end{proof}

\section{Discussion}\label{sec:discussion}

\begin{figure}
\centering
\includegraphics[height=0.3\textwidth]{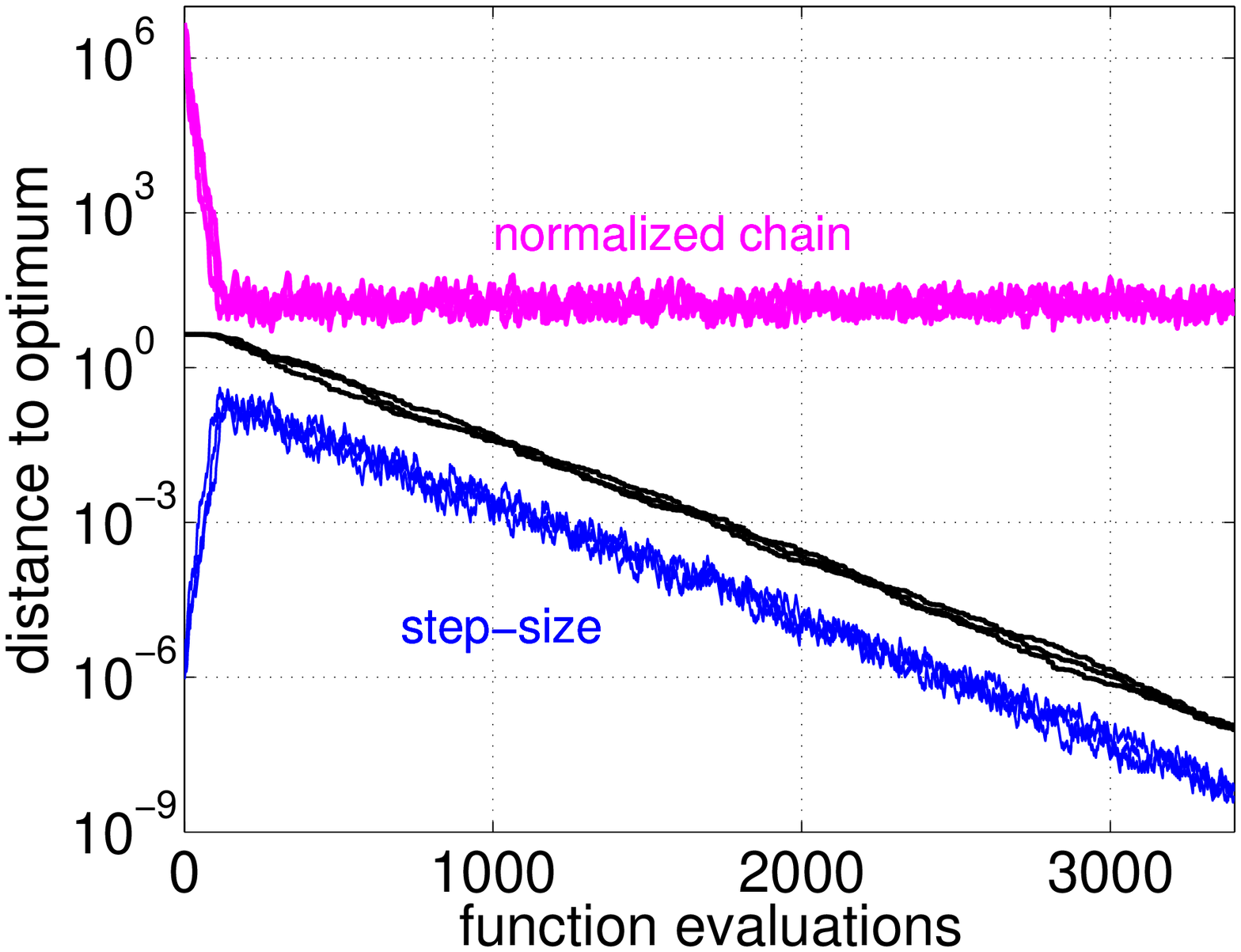}
\includegraphics[height=0.3\textwidth]{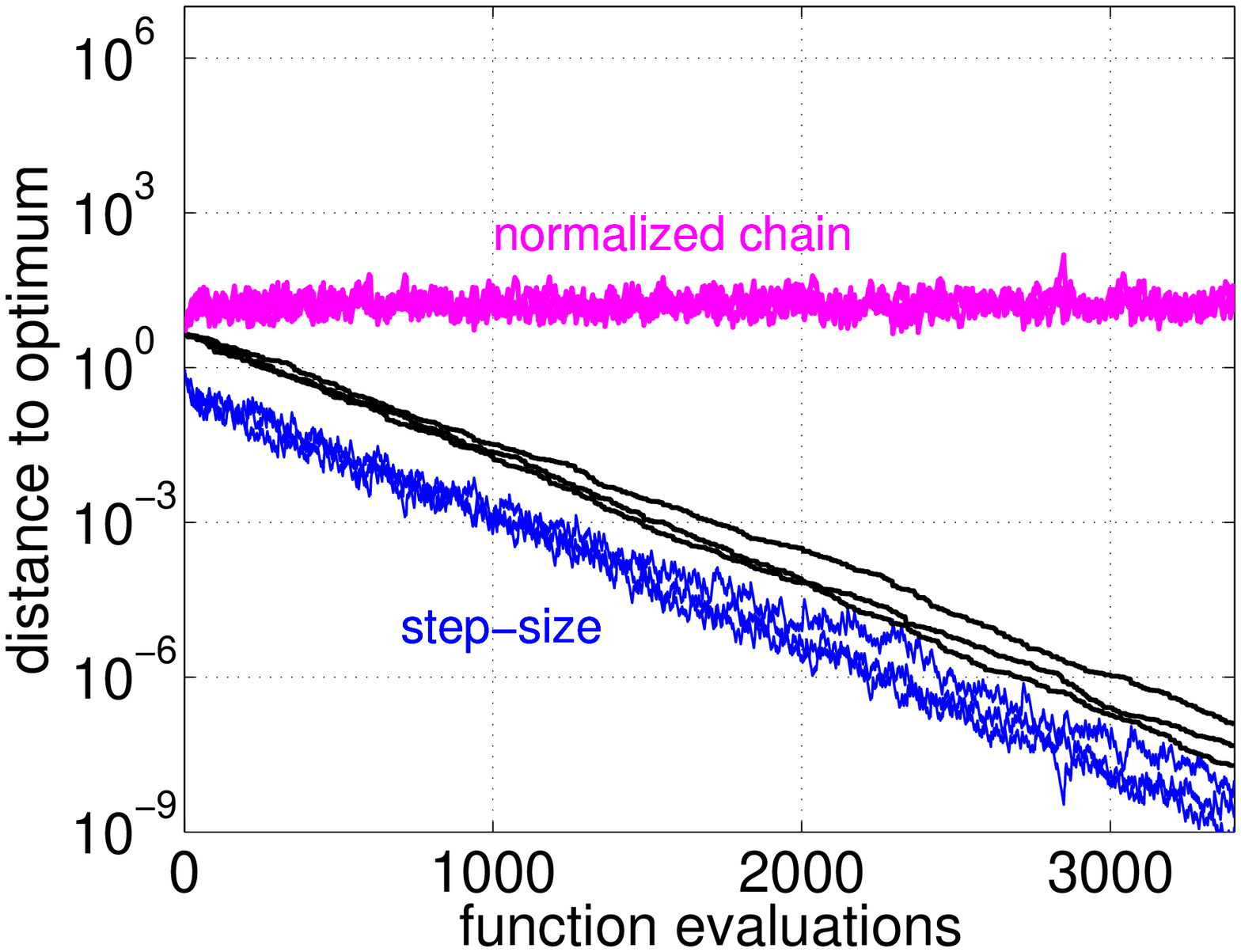}
\caption{\label{fig:simul} Convergence simulations of the $(1+1)$-ES with generalized one-fifth success rule  on spherical functions $f(\x)= g(\| \x \|)$ for $g \in \Monotone$ in dimension $\dim = 20$ (parameters $\gamma=\exp(1/3)$ and $q=4$).
Each plot is in log scale and depicts in black the distance to optimum, i.e., $\| \Xt\|$, in blue the respective step-size $\st$ and in magenta the norm of the normalized chain $\| \Zt\|$. The $x$-axis is the number of function evaluations corresponding to the iteration index $t$. On the left the simulation is voluntarily started with a too small step-size equal to $\sigma_{0}=10^{-6}$ to illustrate the adaptivity ability of the algorithm. On the right, the initial step-size equals $1$.
}
\end{figure}

Using the methodology developed in \citecompanion, we have proven the global linear convergence of the $(1+1)$-ES with generalized one-fifth success rule on functions that write $h = g \circ f $ where $f$ is a continuously differentiable positively homogeneous with degree alpha function (satisfying an additional mild condition on the gradient norm)
and $g$ is a
strictly increasing function. This class of functions includes non quasi-convex functions and non continuous functions, an untypical setting for proving linear convergence of optimization algorithms in general. 

Linear convergence holds under the condition that the step-size increases in case of success, i.e., $\gamma > 1$ and that 
\begin{equation}\label{cond:GE}
\frac12  \left( 1/\gamma^{\alpha} + \gamma^{\alpha/q}  \right)  < 1 \enspace.
\end{equation}
Especially, this condition only depends on the function via $\alpha$ and is thus the same for \emph{any} $g \circ f$ with $f$ is positively homogeneous with degree $\alpha$ and continuously differentiable (plus satisfying Assumptions~\ref{ass:f2}).

Because on a linear function the probability of success equals $1/2$, the condition in \eqref{cond:GE}
 corresponds to the expected inverse of the step-size change to the alpha--on a linear function--being strictly smaller than $1$. In other words, the step-size should increase on a linear function. While this latter condition seems a reasonable requirement for an adaptive step-size algorithm, let us point out that some algorithms like the $(1,2)$-ES with self adaptation fail to satisfy this condition (see \cite{hansen2006ecj} for a thorough analysis of this problem). We believe that the fact that linear convergence on the class of functions investigated in the paper is related to increasing the step-size on linear functions illustrates the strong need to study simple models like the linear function when designing \acprs\ algorithms. 
 
Our statements for the linear convergence hold for any initial solution and any initial step-size. This latter property reflects the main advantage of \emph{adaptive} step-size methods: the initial step-size does not need to be too carefully chosen to ensure good convergence properties. 
Note that methods like simulated annealing or Simultaneous perturbation stochastic approximation (SPSA) \cite{Spall:1992} do not share this nice property and are very sensitive to the choice of some parameters that unfortunately need to be adjusted by the user. 

 The adaptation phase, i.e., how long it takes such that the linear convergence is ``observed'' (see Figure~\ref{fig:simul} left) is related theoretically to the convergence speed of $(\Zt)_{t \in \NNN}$ to its stationary distribution. We have proven a geometric drift that ensures that this convergence is geometrically fast and the geometric rate is independent of the initial condition.

 \todo{Maybe talk about different form of convergence: We have formulated linear convergence in different manner. First almost sure}

Previous attempts to analyze \acprs\ always focused on much smaller classes of functions. The sphere function was analyzed in \cite{TCSAnne04,Jens:2007,jens:gecco:2006}, and a specific class of convex quadratic functions was also analyzed in \cite{jens:2005,jens:tcs:2006}. Our proof is more general: it holds on a wider class of function that also includes convex-quadratic functions. Indeed in Lemma~\ref{lem:CQ} we have seen that convex-quadratic functions satisfy Assumption~\ref{ass:f2} with $\alpha=2$ if $n \geq 2$. Hence linear convergence holds on convex-quadratic functions if $n \geq 2$ under the condition that $\frac12 \left( 1/\gamma^{2} + \gamma^{2/q} \right) < 1$. This latter condition can be relaxed observing that for any $f_{H} ( \x) = \frac 12 \x^{T} H \x$, $f_{H} ( \x) = g_{\alpha} \left( f_{H}^{\alpha/2}(\x) \right)$  for  $g_{\alpha}: x \in \R^{+} \mapsto x^{2/\alpha} $ ($g_{\alpha} \in \Monotone$). The function $\x \mapsto f_{H}^{\alpha/2}(\x) $ is positively homogeneous with degree $\alpha$ and stays continuously differentiable if $\alpha \geq 2$. Hence linear convergence will hold for a given $(\gamma, q) $ if there exists $2 \leq \alpha \leq n$ such that $\frac12 \left( 1/\gamma^{\alpha} + \gamma^{\alpha/q} \right) < 1$.


We have obtained a comprehensive expression for the convergence rate of the algorithm as 
\begin{equation}\label{flora-is-crying}
\CR = - \ln \gamma   \left( \frac{q+1}{q} \PS - \frac1q	  \right) 
\end{equation}
where $\PS$ is the asymptotic probability of success. 
This formula implies that when convergence occurs the probability of success is strictly smaller than $1/(q+1)$, i.e., strictly smaller than $1/5$ using the traditional $1/5$ as target success probability (corresponding to $q=4$). 

While we have proven here that $\CR > 0$, i.e., linear convergence indeed holds, in the case of algorithms that do not guarantee the monotony of $f(\Xt)$, the sign of the ``convergence'' rate is usually not possible to obtain analytically  and one needs to resort to Monte Carlo simulations \cite{TCSAnne04}. 

Besides the sign of the convergence rate, one would like to extract more properties of $\CR$ like the dependence in the dimension, or the dependence in the condition number for convex-quadratic functions. This seems to be hard to achieve with the present approach as $\CR$ depends on the stationary distribution of the normalized chain for which little is known except its existence. However Monte Carlo simulations are natural and always possible to estimate those dependencies. The present paper gives a rigorous framework to perform those Monte Carlo simulations.
Note that using more ad-hoc techniques, it is possible to obtain some dependencies in the dimension or condition number  \cite{Jens:2007,jens:gecco:2006,jens:2005,jens:tcs:2006}.

Though its convergence proof ``resisted'' for more than 40 years, the algorithm analyzed is simple and relatively straightforward as witnessed by the fact that it was already proposed very early and by various researchers in parallel. We however want to emphasize that nowadays this algorithm should mainly have an academic purpose. Indeed more robust comparison-based adaptive algorithm exist, namely the CMA-ES algorithm where in addition to the step-size, a full covariance matrix is adapted \cite{hansen2001}. \todo{Note however that recently, the $(1+1)$-ES was compared with recently introduced methods and compared favorably \todo{detail more - in particular on convexity - read the paper} \cite{Stich}. READ Stich et al. siam paper.}

Last we want to emphasize two points:

1) A common misconception is that randomized methods are good for global optimization and bad for local optimization. The present paper by proving a global linear convergence for a \acprs\ disproves this binary view. In addition, comparisons of the CMA-ES algorithm--the state-of-the-art comparison based adaptive algorithm--with BFGS and NEWUOA show also that CMA-ES is competitive on (unimodal) composite of convex-quadratic functions provided they are significantly non-separable and non-convex \cite{auger:colloquegiens}. This result does not come as a surprise as \acprs\ and CMA-ES were designed first as robust local search and carefully investigated to optimally solve simple functions like the sphere, the linear function and convex-quadratic functions. 
%

 
2) The present paper illustrates that the theory of Markov Chains with discrete time and continuous state space is useful and powerful for the analysis of \acprs. We believe that the present analysis can be extended further for the case of stochastic functions or for the case of algorithms where a covariance matrix is adapted in addition to a step-size.

\nnote{It's a discussion I had I guess a long time ago. It is clear that if the probability of success is 1/5 then the convergence rate is zero. Hence while it is ok to talk about the idea of trying to maintain 1/5 - it does not make sense to have calculate this 1/5 on a function where we want to converge (the sphere) - while we then input that with $1/5$ the step-size should stay constant. }


\newpage

\bibliographystyle{plain}
\bibliography{optimbib}

\end{document}